\documentclass[11pt]{article}
\usepackage{amssymb}
\usepackage{amsmath}
\usepackage{amsthm}
\usepackage[utf8]{inputenc}
\usepackage[T1]{fontenc}
\usepackage{picture}
\usepackage[colorlinks=true,linkcolor=blue]{hyperref}
\usepackage{appendix}
\usepackage{enumitem}
\usepackage{pdfsync}
\usepackage{titlesec}
\usepackage{bbold}
\usepackage[capitalise]{cleveref}
\usepackage{authblk}

\usepackage{pifont}
\usepackage{geometry}
\usepackage[x11names]{xcolor}
\usepackage[colorlinks=true,linkcolor=blue]{hyperref}
\hypersetup{
  colorlinks=true,
  citecolor=SpringGreen4
}

\theoremstyle{plain} 
\newtheorem{prop}{Proposition}[section] 
\newtheorem{thm}[prop]{Theorem}

\newtheorem{lem}[prop]{Lemma}

\theoremstyle{definition}
\newtheorem{defi}[prop]{Definition}
\newtheorem{rem}[prop]{Remark}

\newcommand{\megane}[1]{{}}

\newcommand{\ioe}{\leq}
\newcommand{\soe}{\geq}

\newcommand{\1}{\mathrm{1~\hspace{-1.4ex}l}}

\DeclareMathOperator{\lf}{lf}
\DeclareMathOperator{\hf}{hf}

\DeclareMathOperator{\Span}{Span}
\DeclareMathOperator{\Id}{Id}

\DeclareMathOperator{\J}{\scriptscriptstyle J}

\def\R{\mathbb{R}}
\def\C{\mathbb{C}}
\def\N{\mathbb{N}}
\def\P{\mathbb{P}}
\def\Z{\mathbb{Z}}

\def\eps{\varepsilon}

\def \S{\mathcal{S}}
\def \H{\mathcal{H}_{\J}}

\def \L{\mathcal{L}}
\def \tild{\widetilde}

\title{Local controllability of the bilinear 1D Schrödinger equation with simultaneous estimates}
\author{Mégane Bournissou}

\begin{document}
\maketitle

\begin{abstract}
We consider the 1D linear Schrödinger equation, on a bounded interval, with Dirichlet boundary conditions and bilinear scalar control. The small-time local exact controllability around the ground state was proved in \cite{BL10}, under an appropriate nondegeneracy assumption. Here, we work under a weaker nondegeneracy assumption and we prove the small-time local exact controllability in projection, around the ground state, with estimates on the control (depending linearly on the target) simultaneously in several spaces. These estimates are obtained at the level of the linearized system, thanks to a new result about trigonometric moment problems. Then, they are transported to the nonlinear system by the inverse mapping theorem, thanks to appropriate estimates of the error between the nonlinear and the linearized dynamics. 
\end{abstract}

\tableofcontents

\AtEndDocument{\bigskip{\footnotesize%
\textsc{Univ Rennes, CNRS, IRMAR - UMR 6625, F-35000 Rennes, France.} \par  
\textit{E-mail adress: }{\texttt{megane.bournissou@ens-rennes.fr}}}}

\section{Introduction and main result}
\subsection{Description of the controlled system}
Let $T>0$. In this paper, we consider the 1D Schrödinger equation, 
\begin{equation}  
\label{Schrodinger} \left\{
    \begin{array}{ll}
        i \partial_t \psi(t,x) = - \partial^2_x \psi(t,x) -u(t)\mu(x)\psi(t,x),  \quad (t,x) \in (0,T) \times (0,1),\\
        \psi(t,0) = \psi(t,1)=0, \quad t \in (0,T).
    \end{array}
\right.  \end{equation}
In quantum physics, this equation describes a quantum particle, in an infinite potential well, subjected to an electric field whose amplitude is given by $u(t)$. The dipolar moment, $\mu: (0,1) \rightarrow \mathbb{R}$, depicts the interaction between the electric field and the particle. This equation is a bilinear control system where the state is the wave function $\psi$, such that $\| \psi(t) \|_{L^2(0,1)} = 1$ for all time and $u : (0,T) \rightarrow \mathbb{R}$ denotes a scalar control. 
\subsection{Functional settings}
Unless otherwise specified, in space, we will work with complex valued functions. The Lebesgue space $L^2(0,1)$ is equipped with the classical hermitian scalar product. Let $\mathcal{S}$ be the unit-sphere of $L^2(0,1)$.
The operator $A$ is defined by
\begin{equation*}
D(A):=H^2(0,1) \cap H^1_0(0,1), \quad A\varphi:=-\frac{d^2 \varphi}{dx^2}.
\end{equation*}
Its eigenvalues and eigenvectors are given by
\begin{equation*}
\forall j \in \N^*, \quad \lambda_j:= (j \pi)^2 \quad \text{ and } \quad \varphi_j:=\sqrt{2} \sin(j \pi \cdot).
\end{equation*}
The family of the eigenvectors $(\varphi_j)_{j \in \N^*}$ is an orthonormal basis of $L^2(0,1)$. We also introduce, for all $j \in \N^*$, $\psi_j(t,x):= \varphi_j(x) e^{-i \lambda_j t}$ for $(t,x) \in \R \times [0,1]$, which are solutions of the Schrödinger equation \eqref{Schrodinger} for $u \equiv 0$. When $k=1$, $\psi_1$ is the ground state. We also introduce the normed spaces linked to the operator $A$, given by, for all $s>0$, 
\begin{equation*}
H^s_{(0)}(0,1):=D(A^{\frac{s}{2}}), \quad \|\varphi\|_{H^s_{(0)}(0,1)} := \| \left( \langle  \varphi, \varphi_j\rangle \right)_{j \in \N^*} \|_{h^s}=\left( \sum \limits_{j=1}^{+\infty} | j^s \langle  \varphi, \varphi_j\rangle  |^2 \right)^{\frac{1}{2}}.
\end{equation*}
If $J$ is a subset of $\N^*$, then we define 
\begin{equation*}
\H := \overline{\Span_{\C}} \left( \varphi_j , \ j \in J \right),
\end{equation*}
and we introduce the orthogonal projection on $\H$, given by, 
\begin{equation*}
\begin{array}{ccrcl}
\P_{\J} & : & L^2(0,1) & \to & \H \\
 & & \psi & \mapsto & \psi - \sum \limits_{j \not\in J} \langle \psi, \varphi_j \rangle \varphi_j. \\
\end{array}
\end{equation*}
For $T>0$ and $u \in L^1(0,T)$, the family $(u_n)_{n \in \N}$ of the iterated primitives of $u$ is defined by induction as, 
\begin{equation*}
u_0:=u \quad \text{ and } \quad \forall n \in \mathbb{N}, \ u_{n+1}(t) := \int_0^t u_n(\tau) d\tau, \quad t \in [0,T].
\end{equation*}
We will also consider, for any integer $k \in \N$, $H^k \left( (0,T), \R \right)$, the usual integer-order real Sobolev spaces, equipped with the usual $H^k(0,T)-$norm 
and $H^k_0(0,T)$ the adherence of $C_c^{\infty}(0,T)$, the set of functions with compact support inside $(0,T)$, for the topology $\| \cdot \|_{H^k(0,T)}$. By Poincaré inequality, $H^k_0(0,T)$ can be equipped with the norm
\begin{equation*}
\| u \|_{H^k_0(0,T)}:= \left(  \int_0^T u^{(k)}(t)^2 dt \right)^{1/2}.
\end{equation*}
For any integer $k \in \N^*$, the negative $H^{-k}(0,T)$-norm is not defined by duality as usual but for every $u \in L^2(0,T)$ by 
\begin{equation}
\label{def_norm_faible}
\| u \|_{H^{-k}(0,T)} := | u_1(T) | + \| u_k \|_{L^2(0,T)}, 
\end{equation}
as such norms seem to arise naturally in both the nonlinear and linearized dynamics. For the sake of simplicity, we will sometimes omit $(0,T)$ or $(0,1)$ on the spaces.
\subsection{Main result}
The regularity assumptions play a crucial role in the validity of controllability results. Therefore, we define the following precise notion of small-time local controllability (STLC) used in this paper, stressing the regularity imposed on both the control and the data to be controlled. 
\begin{defi}[STLC around the ground state in $X$ with controls in $Y$]
Let $X$ be a vector space of complex-valued functions defined on $[0,1]$ and $(Y_T, \|\cdot\|_{Y_T})$ be a family of normed vector spaces of real-valued functions defined on $[0,T]$, for $T> 0$. The system \eqref{Schrodinger} is said to be STLC around the ground state in $X$ with controls in $Y$ if for every $T>0$, for every $\varepsilon>0$, there exists $\delta> 0$ such that for every $\psi_*, \psi_f$ in $\mathcal{S}\cap X$ with $\|\psi_* - \psi_1(0)\|_{X} < \delta$ and $\|\psi_f - \psi_1(T)\|_{X} < \delta$,  there exists $u \in L^2(0,T) \cap Y_T$ with $\|u\|_{Y_T} < \varepsilon$ such that the solution $\psi$ of \eqref{Schrodinger} associated with the initial condition $\psi_*$ satisfies $\psi(T)=\psi_f$.  
\end{defi}
\noindent 
Since \cite{BL10}, it is known that if 
\begin{equation}
\label{hyp_mu_old}
\text{there exists a constant } c>0 \text{ such that for all } j \in \N^*, \ | \langle \mu \varphi_1, \varphi_j \rangle | \geq \frac{c}{j^{3}},
\end{equation}
then for any $k \in \N$, the Schrödinger equation \eqref{Schrodinger} is STLC around the ground state in $H^{2k+3}_{(0)}$ with controls in $H^k_0$. However, in \cite{BL10}, the associated control map $(\psi_0, \psi_f) \mapsto u$ depends on $k$. In this article, two goals are tackled:
\begin{itemize}
\item first, building a unique control map for the nonlinear system with simultaneous estimates in various control/data spaces,
\item second, dealing with control in projection when an assumption of the type \eqref{hyp_mu_old} holds only on a subset $J$ of $\N^*$.
\end{itemize}
Our main result is the following one. 
\begin{thm}
\label{thm:contr_lin_proj}
Let $(p, k) \in \N^2$ with $p \soe k$, $J$ a subset of $\N^*$ and $\mu \in H^{2(p+k)+3}( (0,1), \R)$ with $\mu^{(2n+1)}(0)=\mu^{(2n+1)}(1)=0$ for all $n=0, \ldots, p-1$, such that
\begin{equation}
\label{hyp_mu}
\text{there exists a constant } c>0 \text{ such that for all } j \in J, \ | \langle \mu \varphi_1, \varphi_j \rangle | \geq \frac{c}{j^{2p+3}}.
\end{equation}
The Schrödinger equation \eqref{Schrodinger} is STLC in projection around the ground state in $H^{2(p+m)+3}_{(0)}$ with controls in $H^m_0(T_0,T)$, for every $m \in \{0, \ldots, k \}$ with the same control map. 

\bigskip \noindent 
More precisely, for all initial time $T_0 \soe 0$ and final time $T> T_0$, there exists $C$, $\delta >0$ and a $C^1$-map $\Gamma : \Omega_{T_0} \times \Omega_T \rightarrow H^k_0( (T_0,T), \R)$ where 
\begin{align}
\label{def_Omega_T0}
\Omega_{T_0} &:= \{ \psi_0 \in \S \cap H^{2(p+k)+3}_{(0)} ; \ \| \psi_0 - \psi_1(T_0) \|_{H^{2(p+k)+3}_{(0)}} < \delta \},
\\
\label{def_Omega_T}
\Omega_T &:= \{ \psi_f \in \H \cap H^{2(p+k)+3}_{(0)} ; \ \| \psi_f - \P_{\J} \left( \psi_1(T) \right) \|_{H^{2(p+k)+3}_{(0)}} < \delta \},
\end{align}
such that $\Gamma(\psi_1(T_0), \psi_1(T))=0$ and for every $(\psi_0, \psi_f) \in \Omega_{T_0} \times \Omega_T$, the solution of \eqref{Schrodinger} on $[T_0, T]$ with control $u:=\Gamma(\psi_0,\psi_f)$ and initial condition $\psi_0$ at $t=T_0$ satisfies 
$$
 \P_{\J} \left( \psi(T) \right)= \psi_f,
$$
with the following boundary conditions
\begin{equation}
\label{eq:weak_bc_nl}
u_2(T)= \ldots= u_{k+1}(T)=0, 
\end{equation}
where here $(u_n)_{n \in \N}$ denotes the iterated primitives of $u$ vanishing at $T_0$. Besides, for all $m$ in $\{-(k+1), \ldots, k\}$, the following estimates hold
\begin{equation}
\label{estim_contr_nl}
\| u \|_{H^m_0(T_0,T)}
\ioe
C
\left(
\| \psi_0 - \psi_1(T_0) \|_{H^{2(p+m)+3}_{(0)}} 
+
\| \psi_f - \P_{\J} \psi_1(T) \|_{H^{2(p+m)+3}_{(0)}} 
\right)
.
\end{equation}
\end{thm}
To simplify the notations, in all the following, we will take the initial time $T_0$ equal to $0$, the proof when $T_0>0$ is deduced by translation of controls and a change of global phase on the state. Moreover, from now on, if not mentioned, $T>0$ will denote the final time, $p$ and $k$ two integers, and $J$ a subset of $\N^*$. 
\begin{rem}
Assume that $J$ contains an infinite subset of $2\N$ and one of $2\N+1$. Then, for all $\mu$ in $H^{2p+3}(0,1)$ with $\mu^{(2n+1)}(0)=\mu^{(2n+1)}(1)=0$ for all $n=0, \ldots, p-1$, assumption \eqref{hyp_mu} is equivalent to 
\begin{equation*}
\mu^{(2p+1)}(0)\pm \mu^{(2p+1)}(1)\neq 0 \quad \text{ and } \quad \ \forall j \in J, \ \langle \mu \varphi_1, \varphi_j \rangle \neq 0, 
\end{equation*}
as, by integrations by parts and by Riemann-Lebesgue Lemma, 
\begin{equation*}
\langle \mu \varphi_1, \varphi_j \rangle = \frac{(-1)^p2(2p+2)}{\pi^{2p+2} j^{2p+3}} \left( (-1)^{j+1} \mu^{(2p+1)}(1)-\mu^{(2p+1)}(0) \right) + \underset{j \rightarrow +\infty}{o} \left( \frac{1}{j^{2p+3}} \right). 
\end{equation*}
\end{rem}
This result is both a new control result and a toolbox for future works about nonlinear control of the Schrödinger equation \eqref{Schrodinger}. 
Indeed, such result can for example give a framework to prove positive controllability results on the Schrödinger equation, with nonlinear tools, when some of the coefficients $\langle \mu \varphi_1, \varphi_j \rangle$ vanish. In that case, building a unique control map with estimates in simultaneous spaces can be useful to perform specific motions for the nonlinear solution. 
The proof of \cref{thm:contr_lin_proj} is in three steps. 
\begin{itemize}
\item In \cref{well-posedness}, we study the well-posedness of the Schrödinger equation and more precisely the regularity of the solutions with respect to the boundary conditions on the dipolar moment $\mu$.
\item In \cref{subsection:moment}, we present a new result about the solvability of trigonometric moment problems in high regularity spaces with simultaneous estimates.
\item This new moment result allows in \cref{main_theorem} and more precisely in Subsection \ref{C1-regularity} to build a linear control operator, for the linearized system around the ground state, with simultaneous estimates in various control/data spaces. Then, in Subsection \ref{main_proof}, we prove that the iterations of the inverse mapping theorem propagate these estimates to the nonlinear control operator of \eqref{Schrodinger}.
\end{itemize}

\begin{rem}
Actually, the question of building a control function that inherits the regularity of the data to be controlled has already been tackled by Ervedoza and Zuazua in \cite{EZ10}, for time-reversible linear systems. As the use of \cite{EZ10} is not straightforward in our case, we choose to present in this article a new result about trigonometric moment problems solving this question. However, in Section \ref{papier_Ervedoza}, we also explain how the controllability of the linearized system with simultaneous estimates can be proved using the results of \cite{EZ10}. 
\end{rem}

\subsection{Bibliography}

\paragraph{Local exact controllability results.} From a general negative result on the controllability of bilinear control systems by Ball, Marsden and Slemrod \cite{BMS82},  Turinici in \cite{T00} deduced a negative control result for the Schrödinger equation \eqref{Schrodinger}: for a given initial data $\psi_0 \in H^2_{(0)}(0,1) \cap \S$, the reachable set with controls in $L^r_{loc}( (0, +\infty), \R)$, with $r>1$, has an empty interior in $H^2_{(0)}(0,1) \cap \S$. The case of controls in $L^1_{loc}( (0, +\infty), \R)$ has been proved later in \cite{BCC20} by Boussaid, Caponigro, and Chambrion.  

However, choosing more appropriate functional spaces, exact local controllability results for 1D models have been proved by Beauchard in \cite{B05, B08}, whose proofs have been later simplified by Beauchard and Laurent in \cite{BL10} by means of a hidden regularizing effect and an inverse mapping theorem (instead of Nash-Moser's one). 

This strategy was later developed by Morancey and Nersesyan to control one Schrödinger equation with a polarizability term \cite{MN14} or a finite number of equations with one control \cite{M14, MN15}. This was also used by Puel \cite{P16} to prove the local exact controllability for a Schrödinger equation, in a bounded regular domain, in a neighborhood of an eigenfunction corresponding to a simple eigenvalue in dimension $N \leq 3$. 

\paragraph{Global approximate results.} With geometric techniques for the controllability of the Galerkin approximations, in \cite{CMSB09} Chambrion, Mason, Sigalotti, and Boscain prove the approximate controllability of Schrödinger in $L^2$ under hypotheses later refined by Boscain, Caponigro, Chambrion, and Sigalotti in \cite{BCCS12}. In higher order Sobolev spaces, similar results were proved for one \cite{BCC20} or a finite number of equations \cite{BCS14}. Such types of results can also be proved from exact controllability results in infinite time \cite{NN12bis} or from a variational argument \cite{N09}. 

\paragraph{About smooth controllability.} The negative controllability result \cite{T00} and the positive controllability results \cite{B05, BL10} proved on the same Schrödinger equation underline the importance of the regularity assumptions asked for the validity of controllability results, for a linear infinite dimensional equation. Nevertheless, even when the state lives within a finite dimensional space, Beauchard and Marbach proved in \cite{BM18} that the same nonlinear system, according to the functional setting, may or may not be small-time controllable. The same authors highlighted later the same phenomenon on a nonlinear infinite dimension parabolic equation in \cite{BM20}.

Moreover, for a controlled system that is already known to be controllable in a given setting, one can ask whether the control map preserves the smoothness of the data to be controlled. More precisely, if the data is smoother than expected, does the control constructed inherit from this smoothness? Generally, it is not the case. However, \cite{EZ10} gave a method to compute such controls. Such a question can be relevant to deal with nonlinear problems as in \cite{DL09} or to compute convergence rates for numerical approximation (see \cite{BFO20} for example).  

\paragraph{About moment problems.}
The use of infinite moment theory for linear control problems was introduced by the work of Fattorini and Russel (see \cite{FR71, FR75}). For classical results about Riesz basis and moment problems, the reader can refer, for example, to the following works: Krabs in \cite{K92},  Avdonin and  Ivanov in \cite{AI95}, Komornik and Loreti in \cite{KL05}, Haraux in \cite{H89}. In this paper, the solvability of a moment problem at any time with estimates in simultaneous Sobolev spaces is investigated. For the Schrödinger equation, it has already been done by Beauchard in \cite{B05} but only for a specific and not arbitrary small time. Moreover, the solvability of a moment problem in high-regularity spaces, but without simultaneous estimates, has been done in \cite{BL10} for Schrödinger, or in \cite{BM20} for a parabolic equation, relying on the work \cite{BBBO14}.

\section{Well-posedness of the Cauchy problem}
\label{well-posedness}
This section is dedicated to the proof of the existence, uniqueness and bounds on the solution of the Cauchy problem
\begin{equation}  
\label{Schro_source_term} 
\left\{
    \begin{array}{ll}
        i \partial_t \psi(t,x) = - \partial^2_x \psi(t,x) -u(t)\mu(x)\psi(t,x) -f(t,x),  \quad (t,x) \in (0,T) \times (0,1),\\
        \psi(t,0) = \psi(t,1)=0, \quad t \in (0,T), \\
        \psi(0,x) = \psi_0(x), \quad x \in (0,1).
    \end{array}
\right.  
\end{equation}
Our goal is to underline the link between the regularity of the solutions and the boundary conditions on the dipolar moment $\mu$ by proving the following statement. %
\begin{thm}
\label{thm:well-posedness}
Let $T>0$, $(p,k) \in \N^2$, $\mu \in H^{2(p+k)+3}( (0,1), \R)$ with $\mu^{(2n+1)}(0)=\mu^{(2n+1)}(1)=0$ for all $n=0, \ldots, p-1$ , $u \in H^{k}_0( (0,T), \R)$, $\psi_0 \in H^{2(p+k)+3}_{(0)}(0,1)$ and $f \in H^{k}_0 ( (0,T), H^{2p+3} \cap H^{2p+1}_{(0)}(0,1))$.
There exists a unique weak solution of the Schrödinger equation, that is a function $\psi \in C^{k}( [0,T], H^{2p+3}_{(0)}(0,1))$ such that the following equality holds in $H^{2p+3}_{(0)}$ for every $t \in [0,T]$:
\begin{equation*}
\psi(t) = e^{-iAt} \psi_0 + i \int_0^t e^{-iA (t- \tau)} \left( u(\tau) \mu \psi(\tau) + f(\tau) \right) d\tau.
\end{equation*}
Moreover, for every $R>0$, there exists $C=C(T, \mu, R)>0$ such that if $\| u \|_{H^k_0} < R$, then this solution satisfies
\begin{equation}
\label{estim_sol}
\| \psi \|_{C^k( [0,T], H^{2p+3}_{(0)})} 
\ioe 
C
\left(
\| \psi_0 \|_{H^{2(p+k)+3}_{(0)}} 
+
\| f \|_{H^k( (0,T), H^{2p+3} \cap H^{2p+1}_{(0)})}
\right).
\end{equation}
\end{thm}
We will sometimes write $\psi( \cdot ; \ u, \psi_0)$ to denote the solution of \eqref{Schro_source_term} associated with control $u$ and initial data $\psi_0$ when we will need to recall the dependence with respect to the control or the initial condition. For $p=0$, no boundary conditions are needed on $\mu$. The proof of \cref{thm:well-posedness} is inspired by \cite[Proposition 2 and 5]{BL10}, where the authors dealt with the cases $(p=0, k=0)$ and $(p=0, k=1)$.

\begin{rem}
\label{rem:point_final}
Let $T$, $p$, $k$, $\mu$, $u$, $\psi_0$ and $f$ as in \cref{thm:well-posedness}. Notice that as both the control and the source term vanish at the final time, the solution $\psi \in C^k( [0,T], H^{2p+3}_{(0)})$ of \eqref{Schro_source_term} satisfies the following equality in $H^{2p+3}_{(0)}(0,1)$
\begin{equation*}
i \partial_t^k \psi(T) = A^k \psi(T).
\end{equation*}
Therefore,  
\begin{equation*}
\psi(T) \in H^{2(p+k)+3}_{(0)}(0,1).
\end{equation*}
However, in general, the solution does not belong to $ C^0( [0,T], H^{2(p+k)+3}_{(0)})$, as $\psi(t)$ does not belong to $H^{2(p+k)+3}_{(0)}(0,1)$ if $u(t)\neq 0$.  Moreover, from \eqref{estim_sol}, one deduces that for every $R>0$, there exists $C=C(T, \mu, R)>0$ such that if $\| u \|_{H^k_0} < R$, then
\begin{equation}
\label{estim_sol_fin}
\| \psi(T) \|_{H^{2(p+k)+3}_{(0)}}
\ioe
C
\left(
\| \psi_0 \|_{H^{2(p+k)+3}_{(0)}} 
+
\| f \|_{H^k( (0,T), H^{2p+3} \cap H^{2p+1}_{(0)})}
\right).
\end{equation}
\end{rem}

\subsection{Smoothing effect}
The main difficulty in this well-posedness result relies on the fact that, for a given $\tau$ in $[0, T]$, $f(\tau)$ is not assumed to belong to $H^{2p+3}_{(0)}$ and moreover, the operator $\psi \mapsto \mu \psi$ is not bounded from $H^{2p+3}_{(0)}$ to $H^{2p+3}_{(0)}$ because $\mu^{(2p+1)}$ does not vanish at $x=0$ and $x=1$. Therefore, the proof of \cref{thm:well-posedness} stems from the regularity of the function $t \mapsto \int_0^t e^{iA\tau} f(\tau) d\tau$ in the spatial space $H^{2p+3}_{(0)}$ even when $f$ does not take values in such space. In the following proposition, for $p=-1$, we will write $H^{2p+3} \cap H^{2p+1}_{(0)}$ to denote only the space $H^1$ to homogenize with the notation when $p \in \N$. 
\begin{prop}
\label{estim_G_C0}
Let $p=-1$ or $p \in \N$. There exists a nondecreasing function $C: [0, +\infty) \rightarrow (0, +\infty)$ such that for all $T \soe0$  and for all $f \in L^2((0,T), H^{2p+3} \cap H^{2p+1}_{(0)}(0,1))$, the function $G : t \mapsto \int_0^t e^{-iA(t-\tau)} f(\tau) d\tau$ belongs to $C^0( [0,T], H^{2p+3}_{(0)}(0,1))$ with the following estimate, 
\begin{equation}
\| G \|_{C^0( [0,T], H^{2p+3}_{(0)})} \ioe C(T) \| f \|_{L^2((0,T), H^{2p+3} \cap H^{2p+1}_{(0)})}. 
\end{equation}
\end{prop}

\begin{proof}
Let $T \soe 0$ and $f \in L^2((0,T), H^{2p+3} \cap H^{2p+1}_{(0)})$. Let $t \in [0,T]$. By definition of the function $G$ and of the norm $H^{2p+3}_{(0)}$, one seeks to estimate, 
\begin{equation*}
\| G(t) \|_{H^{2p+3}_{(0)}}  
= 
\left\|
\sum \limits_{j=1}^{+\infty} 
\left( 
\int_0^t
\langle 
f(\tau), 
\varphi_j 
\rangle 
e^{-i \lambda_j (t-\tau)}
d\tau
\right) 
\varphi_j
\right\|_{H^{2p+3}_{(0)}}  
=
\left\| 
\left(
\int_0^t \langle 
f(\tau), 
\varphi_j 
\rangle 
e^{i \lambda_j \tau}
d\tau 
\right)_{j \in \N^*}
\right\|_{h^{2p+3}}.
\end{equation*}
Yet, for almost every $\tau \in (0,T)$, $f(\tau)$ belongs to $H^{2p+3} \cap H^{2p+1}_{(0)}$. Therefore, performing $(2p+3)$-integrations by parts, we get, for all $j \in \N^*$, 
\begin{equation*}
\langle f(\tau), \varphi_j \rangle 
=
\frac{\sqrt{2}}{(j \pi)^{2p+3}} 
\left( 
(-1)^j \partial_x^{2p+2} f(\tau,1) 
-
\partial_x^{2p+2} f(\tau,0) 
\right)  
-
\frac{1}{(j \pi)^{2p+3}} 
\langle \partial_x^{2p+3} f(\tau), \sqrt{2} \cos(j \pi \cdot) \rangle,
\end{equation*}
(with a minus added on each term if $p=-1$).
Thus, there exists a constant $C>0$, not depending on time, such that
\begin{multline}
\label{estim_G(t)}
\| G(t) \|_{H^{2p+3}_{(0)}} 
\ioe 
C 
\sum 
\limits_{x_0 \in \{0,1\}}
\left\| 
\left(
\int_0^t 
\partial_x^{2p+2} f(\tau,x_0) 
e^{i \lambda_j \tau}
d\tau
\right)_{j \in \N^*}
\right\|_{l^2} 
\\+
C
\left\| 
\left(
\int_0^t \langle 
\partial_x^{2p+3} f(\tau), 
\sqrt{2} \cos(j \pi \cdot)
\rangle 
e^{i \lambda_j \tau}
d\tau 
\right)_{j \in \N^*}
\right\|_{l^2} 
.
\end{multline}
Using the Cauchy-Schwarz inequality (in time) and then the orthonormality of the family $(\sqrt{2} \cos(j \pi \cdot))_{j \in \N}$ in $L^2(0,1)$, the square of the last term of the right-hand side of \eqref{estim_G(t)} is estimated by
\begin{align*}
\sum \limits_{j=1}^{+\infty}
\left|
\int_0^t \langle 
\partial_x^{2p+3} f(\tau), 
\sqrt{2} \cos(j \pi \cdot)
\rangle 
e^{i \lambda_j \tau}
d\tau 
\right|^2
&\ioe 
\sum \limits_{j=1}^{+\infty}
t
\int_0^t 
\left|
\langle 
\partial_x^{2p+3} f(\tau), 
\sqrt{2} \cos(j \pi \cdot)
\rangle 
\right|^2
d\tau 
\\
&\ioe 
t
\int_0^t 
\left\|
\partial_x^{2p+3} f(\tau)
\right\|_{L^2}^2 
d\tau
,
\end{align*}
giving that 
\begin{equation}
\label{eq:G(t)_intern}
\left\| 
\left(
\int_0^t \langle 
\partial_x^{2p+3} f(\tau), 
\sqrt{2} \cos(j \pi \cdot)
\rangle 
e^{i \lambda_j \tau}
d\tau 
\right)_{j \in \N^*}
\right\|_{l^2} 
\ioe
\sqrt{t} 
\| f \|_{L^2( (0,t), H^{2p+3})}.
\end{equation}
Moreover, the sum in the right-hand side of \eqref{estim_G(t)} is estimated using an Ingham inequality (see for example \cite[Appendix B, Corollary 4]{BL10}) which gives the existence of a nondecreasing function $C: t \mapsto C(t)>0$ such that, 
\begin{equation}
\label{eq:G(t)_bc}
\left\| 
\left(
\int_0^t 
\partial_x^{2p+2} f(\tau,x_0) 
e^{i \lambda_j \tau}
d\tau 
\right)_{j \in \N^*}
\right\|_{l^2} 
\ioe C(t)
 \| \partial_x^{2p+2} f( \cdot ,x_0) \|_{L^2(0,t)}, \quad \text{ for } x_0=0 \text{ and } 1.
\end{equation}
Therefore, \eqref{estim_G(t)}, \eqref{eq:G(t)_intern} and \eqref{eq:G(t)_bc} together with the fact that $H^{2p+3}(0,1)$ is continuously embedded in $C^{2p+2}([0,1])$ give
\begin{equation*}
\| G(t) \|_{H^{2p+3}_{(0)}} \ioe C(t) \| f \|_{L^2((0,t), H^{2p+3})}, 
\end{equation*}
with a nondecreasing function $t \mapsto C(t)>0$. This bound shows that $G(t)$ belongs to $H^{2p+3}_{(0)}(0,1)$ for every $t \in [0,T]$ and that the map $t \mapsto G(t) \in H^{2p+3}_{(0)}$ is continuous at $t=0$ (as $C(t)$ is uniformly bounded when $t \rightarrow 0$ and $\| f \|_{L^2((0,t), H^{2p+3})} \rightarrow 0$ when $t \rightarrow 0$ thanks to the dominated convergence theorem). The continuity of $G$ at any time $t \in (0,T]$ can be proved similarly. 
\end{proof}
The previous lemma stated the continuity of $t \mapsto \int_0^t e^{iA \tau} f(\tau) d\tau$ and from this we can deduce, for all $k \in \N^*$, the $C^k$-regularity of such function.
\begin{prop}
\label{estim_G_Ck}
Let $(p, k) \in \N^2$. There exists a nondecreasing function $C : [0, +\infty) \rightarrow (0, +\infty)$ such that for all $T \soe 0$ and for all $f \in H^{k}_0((0,T), H^{2p+3} \cap H^{2p+1}_{(0)}(0,1))$, the function $G: t \mapsto \int_0^t e^{-iA(t-\tau)} f(\tau) d\tau$ belongs to $C^k( [0,T], H^{2p+3}_{(0)}(0,1))$ with the following estimate,  
\begin{equation}
\| G \|_{C^k( [0,T], H^{2p+3}_{(0)})} \ioe C \| f \|_{H^k((0,T), H^{2p+3} \cap H^{2p+1}_{(0)})}.
\end{equation}
\end{prop}

\begin{proof}
Let $f \in H^k_0((0,T), H^{2p+3} \cap H^{2p+1}_{(0)})$. We will rather work with $G$ written under the form
\begin{equation*}
G(t) =  \int_0^t e^{-iA \tau} f(t - \tau) d\tau, \quad t \in [0,T].  
\end{equation*}
\noindent \textit{Step 1: Classical regularity.}
As $f$ is in $H^k((0,T), H^{2p+2}_{(0)})$, the classical theory on semi-groups gives that $G$ is in $C^k( [0,T], H^{2p+2}_{(0)}).$ Moreover, because $f(0)=\ldots=f^{(k-1)}(0)=0$, the derivatives, for the $H^{2p+2}_{(0)}$-topology, are given by,
\begin{equation}
\label{expr_der_G}
\forall n=0, \ldots, k, \ G^{(n)}(t) = \int_0^t e^{-iA\tau} f^{(n)}(t-\tau) d\tau, \quad t \in [0,T].
\end{equation}
\noindent \textit{Step 2: Higher regularity in space.}
We prove that $G$ is in $C^n( [0,T], H^{2p+3}_{(0)})$ by induction on $n \in \{0, \ldots k\}$. The initialization ($n=0$) is proved in \cref{estim_G_C0}. Let $n \in \{0, \ldots k-1\}$ and assume that $G$ is in $C^n( [0,T], H^{2p+3}_{(0)})$.  First, as $f^{(n+1)}$ is in $L^2((0,T),  H^{2p+3} \cap H^{2p+1}_{(0)})$ (as $n+1\leq k$), \cref{estim_G_C0} and \eqref{expr_der_G} give directly that $G^{(n+1)}$ is in $C^0( [0,T], H^{2p+3}_{(0)})$.  Then, for $t \in [0,T]$, 
%
with \eqref{expr_der_G}, one can write, 
\begin{multline}
\label{eq:taux_accroissement}
\frac{G^{(n)}(t+h)-G^{(n)}(t)}{h}  -  G^{(n+1)}(t)
 =
\frac{1}{h} 
\int_t^{t+h} 
e^{-iA \tau} 
f^{(n)}(t+h- \tau) 
d\tau 
\\ +
\int_0^t
e^{-iA \tau} 
\left(
\frac{f^{(n)}(t+h- \tau)-f^{(n)}(t-\tau)}{h}-f^{(n+1)}(t-\tau)
\right)
d\tau.
\end{multline}
By \cref{estim_G_C0}, the $H^{2p+3}_{(0)}$-norm of the second term of the right-hand side of \eqref{eq:taux_accroissement} is bounded by 
\begin{equation*}
c_1(T)
\left\| 
\frac{f^{(n)}( \cdot +h)-f^{(n)}}{h}
-
f^{(n+1)}
\right\|_{L^2( (0,T), H^{2p+3} \cap H^{2p+1}_{(0)})}
\end{equation*}
which goes to zero as $h$ goes to zero, because $f^{(n)}$ is in $H^1((0,T), H^{2p+3} \cap H^{2p+1}_{(0)})$. Besides, as $f^{(n)}(0)=0$,  using successively a change of variables, that $e^{-iA}$ is an isometry from $H^{2p+3}_{(0)}$ to $H^{2p+3}_{(0)}$, \cref{estim_G_C0} and Cauchy-Schwarz inequality, one gets the following upper bound for the $H^{2p+3}_{(0)}$-norm of the first term of the right-hand side of \eqref{eq:taux_accroissement}
\begin{align*}
\Big\| 
&\frac{e^{iA(t+h)}}{h} 
\int_0^{h} 
e^{-iAs} 
\left(
f^{(n)}(s) - f^{(n)}(0)
\right)
ds
\Big\|_{H^{2p+3}_{(0)}}
\ioe 
c_1(h)
\left\|
\frac{f^{(n)}( \cdot) - f^{(n)}(0)}{h}
\right\|_{L^2( (0,h), H^{2p+3} \cap H^{2p+1}_{(0)})}
\\
&= \frac{c_1(h)}{h} 
\left\|
\int_0^{\cdot}
\partial_t f^{(n)}(\tau) d\tau
\right\|_{L^2( (0,h), H^{2p+3} \cap H^{2p+1}_{(0)})}
\ioe c_1(h) 
\left\|
\partial_t f^{(n)}
\right\|_{L^{2}( (0,h), H^{2p+3} \cap H^{2p+1}_{(0)})},
\end{align*}
This bound goes to zero when $h$ goes to zero as the function $h \mapsto c_1(h)$ given in \cref{estim_G_C0} is nondecreasing and 
$\lim \limits_{h \rightarrow0}
\left\|
\partial_t f^{(n)}
\right\|_{L^{2}( (0,h), H^{2p+3} \cap H^{2p+1}_{(0)})}=0
$
by the dominated convergence theorem. And this concludes the proof. 
\end{proof}

\subsection{Proof of \cref{thm:well-posedness}, the well-posedness}
Let $\mu$, $\psi_0$, $f$ and $u$ satisfying the hypotheses of \cref{thm:well-posedness}. We consider the map 
\begin{equation*}
\begin{array}{ccrcl}
F & : & C^k ( [0,T], H^{2p+3}_{(0)}(0,1) )  & \to & C^k ( [0,T], H^{2p+3}_{(0)}(0,1)) \\
 & & \psi & \mapsto & \xi, \\
\end{array}
\end{equation*}
where 
\begin{equation*}
\xi(t) := e^{-iAt } \psi_0 + i\int_0^t e^{-iA (t-\tau)} \left( u(\tau) \mu \psi(\tau) + f(\tau) \right) d\tau, \quad t \in [0,T], 
\end{equation*}
so that $\psi$ is a weak solution of \eqref{Schro_source_term} if and only if $\psi$ is a fixed-point of $F$. 

\medskip \noindent \textit{$F$ is well-defined.}
For $\psi \in C^k ( [0,T], H^{2p+3}_{(0)})$ and $\tau \in [0,T]$, by Leibniz formula and the algebra structure of $H^{2p+3}$, the map $\tau \mapsto  u(\tau) \mu \psi(\tau) + f(\tau)$ belongs to $ H^k_0( (0,T), H^{2p+3} \cap H^{2p+1}_{(0)})$ and thus, by \cref{estim_G_Ck}, $\xi=F(\psi)$ is in $C^k ( [0,T], H^{2p+3}_{(0)})$. 

\medskip \noindent \textit{$F$ is a contraction.}
Let $\psi, \hat{\psi}$ in $C^k ( [0,T], H^{2p+3}_{(0)})$. Again, by \cref{estim_G_Ck} and the algebra structure of $H^{2p+3}(0,1)$, we get, for all $t \in [0,T]$, 
\begin{equation}
\label{F_contraction}
\| F(\psi)(t) - F(\hat{\psi})(t) \|_{H^{2p+3}_{(0)}}
\ioe C(T)
\| u \|_{H^k} \| \mu \|_{H^{2p+3}} \| \psi - \hat{\psi} \|_{C^k ( [0,T], H^{2p+3})}. 
\end{equation}
Equation \eqref{F_contraction} proves that if $\| u \|_{H^k}$ is small enough, $F$ is a contraction and thus by the Banach fixed-point theorem, admits a unique fixed-point $\psi$ in $C^k ( [0,T], H^{2p+3}_{(0)})$. Computing the same estimates, by \cref{estim_G_Ck}, we get that this fixed point satisfies
\begin{multline*}
\| \psi\|_{C^k ( [0,T], H^{2p+3}_{(0)} ) }
\ioe 
C(T, \mu) 
\Big(
\| \psi_0\|_{H^{2(p+k)+3}_{(0)}}
 +
\| u \|_{H^k}
\| \psi\|_{C^k ( [0,T], H^{2p+3}_{(0)} ) }
\\ +
\| f\|_{H^k( (0,T), H^{2p+3} \cap H^{2p+1}_{(0)})}
\Big).
\end{multline*}
Therefore, for $u \in H^k_0(0,T)$ such that $C(T, \mu) \| u \|_{H^k} \ioe 1/2$, we get \eqref{estim_sol}. If $u$ in $B_{H^k(0,T)}(0,R)$ is not small enough in $H^k_0$, one can consider a subdivision $0 = T_0 < \ldots < T_N=T$ such that for all $i \in \{0, \ldots, N-1\}$,  $\| u \|_{H^k(T_i,T_{i+1})}$ is small enough to apply the previous argument on $[T_i,T_{i+1}]$. Notice that as the constant $T \mapsto C(T)$ is nondecreasing, $N$ only depends on $R$ so that the constant in \eqref{estim_sol} does only depend on $T$, $\mu$ and $R$, as claimed in the theorem.

\section{Solvability of a moment problem with simultaneous estimates}
\label{subsection:moment}
The goal of this section is the proof of a new result about trigonometric moment problems. More precisely, the aim is to prove the solvability of a moment problem in high-regularity spaces, with simultaneous estimates on the operator solving the moment problem. This will allow to build a control map for the linearized system with estimates in various control/data spaces in the following section. 

\subsection{Assumptions on the frequencies}
Given an increasing sequence $\omega=(\omega_j)_{j \in \N}$ of $[0,+\infty)$ with $\omega_0=0$, we define, for all $m \in \N$,
\begin{equation*} 
h^{2m}_{\omega,r}(\N, \C) := \Big\{ (d_j)_{j \in \N} \in l^2(\N, \C) ; \ d_0 \in \R \text{ and } \ \| d \|^2_{h^{2m}_{\omega}}:=\sum \limits_{j=0}^{+\infty} \left| \left(\delta_{j,0}+ \omega_j^{m} \right) d_j \right|^2 < + \infty \Big\}.
\end{equation*}
When $m=0$, we will simply write $l_r^2$ instead of $h^0_{\omega, r}$. Moreover, we will say that a sequence $(\omega_j)_{j \in \N}$ satisfies an asymptotic gap if
\begin{equation*}
\tag{AsymptGap}
\label{AsymptGap}
\omega_{j+1} - \omega_j \rightarrow +\infty \text{ when } j \rightarrow +\infty, 
\end{equation*}
and satisfies a polynomial asymptotic gap if
\begin{equation*}
\tag{AsymptGapPoly}
\label{AsymptGapPoly}
\text{there exists } \eps>0, N_0 \in \N \text{ and } c>0 \text{ s.t.\ for all } j\soe N_0, \quad \omega_{j+1} - \omega_j \soe cj^{\eps}. 
\end{equation*} 
If not mentioned, for all $j \in \Z$, $j <0$, we denote by $\omega_{j}=-\omega_{-j}$. 

\subsection{Solvability of a moment problem in $L^2(0,T)$ with polynomial constraints}
First, following some known results about trigonometric moment problems, one can prove the solvability in high-regularity spaces of such problems but without simultaneous estimates. This can be deduced from the solvability of a moment problem in $L^2(0,T)$ with polynomial constraints, which is therefore the starting point of this section. The results presented in this subsection are a generalization of the work \cite[Appendix B]{BL10}. Therein, the full proofs are left to the reader. 

First, one can state that under an asymptotic gap condition, the family of complex exponentials, with an added finite number of polynomials, has a biorthogonal family.
\begin{lem}
Let $T>0$ and $(\omega_j)_{j \in \N}$ an increasing sequence of $[0, +\infty)$ such that $\omega_0=0$ and satisfying \eqref{AsymptGap}.Then, for all $n \in \N^*$, the family $\Theta_n:= \left\{ t^q ; \ q=1, \ldots n \right\} \cup \left\{ \ e^{i \omega_j t} ; \ j \in \Z \right\}$ is minimal in $L^2(0,T)$ and thus admits a biorthogonal family. 
\end{lem}
From this result and the work of Haraux \cite{H89}, one can state the solvability of trigonometric moment problem in $L^2(0,T)$ with a finite number of constraints on polynomial moments.
\begin{thm}
\label{thm_mom_L2_poly}
Let $T>0$, $n \in \mathbb{N}^*$ and $(\omega_j)_{j \in \mathbb{N}}$ an increasing sequence of $[0, +\infty)$ such that $\omega_0=0$ and satisfying \eqref{AsymptGap}. 
There exists a constant $C>0$ and a continuous linear map $\L^T_0 : \R^n \times l^{2}_{r}(\N, \C) \rightarrow L^2((0,T), \R)$ such that for every sequence $d=\left( (d_{-q})_{q=1, \ldots, n}, \ (d_j)_{j \in \mathbb{N}} \right) \in \R^n \times l^{2}_{r}( \mathbb{N} , \mathbb{C})$, the control $u:=\L^T_0(d) \in L^2( (0,T), \R)$ satisfies the moment problem 
\begin{equation*}
\forall j \in \mathbb{N}, 
\
\int_0^T u(t) e^{i \omega_j t} dt
=
d_j
\quad 
\text{ and }
\quad
\forall q=1, \ldots, n,
\
\int_0^T t^q u(t) dt = d_{-q},
\end{equation*}
and the following size estimate
$$
\| u \|_{L^2(0,T)} \ioe C \left( \sum \limits_{j=-n}^{+\infty} |d_j|^2 \right)^{1/2}.
$$
\end{thm}
Then, we deduce, for any integer $k$, the solvability of such moment problems in $H^k_0(0,T)$ with only an estimate in the most regular space. 
\begin{thm}
\label{thm_mom_Hm}
Let $T>0$, $(n, k) \in (\N^*)^2$ and $(\omega_j)_{j \in \mathbb{N}}$ an increasing sequence of $[0, +\infty)$ such that $\omega_0=0$ and satisfying \eqref{AsymptGap}.
There exists a constant $C>0$ and a continuous linear map $\tild{\L}^T_k : \R^n \times h^{2k}_{\omega,r}(\N, \C) \rightarrow H^k_0((0,T), \R)$ such that for every sequence $d=\left( (d_{-q})_{q=1, \ldots, n}, \ (d_j)_{j \in \mathbb{N}} \right) \in \R^n \times h^{2k}_{\omega,r}( \mathbb{N} , \mathbb{C})$, the control $u:=\tild{\L}^T_k(d) \in H^k_0( (0,T), \R)$ satisfies the moment problem 
\begin{equation*}
\forall j \in \mathbb{N}, 
\
\int_0^T u(t) e^{i \omega_j t} dt
=
d_j
\quad 
\text{ and }
\quad
\forall q=1, \ldots, n,
\
\int_0^T t^q u(t) dt = d_{-q},
\end{equation*}
and the following size estimate
$$
\| u \|_{H^k_0(0,T)} 
\ioe 
C
\left(
\sum \limits_{j=-n}^{+\infty} \left| \left( \delta_{j,0} + \omega_j^k \right) d_j \right|^2
\right)^{1/2}
.
$$
\end{thm}
\begin{proof}
The proof follows with $\tild{\L}^T_k := \mathcal{B}_k \circ \L_0^T \circ \mathcal{A}_k$ where $\L_0^T$ is defined in \cref{thm_mom_L2_poly},
$\mathcal{A}_k : \R^n \times h^{2k}_{\omega,r}( \mathbb{N} , \mathbb{C})  \to  \R^{k+n} \times l_r^2(\N, \C)$ and $\mathcal{B}_k : L^2( (0,T), \R) \to  H^k((0,T), \R)$ are respectively given by
\begin{equation*}
\mathcal{A}_k
\left( (d_{-q})_{q=1, \ldots,n}, \ (d_j)_{j \in \N} \right)
:=
\left( 
\left(
\frac{(-1)^k q!}{(q-k)!} d_{-q+k} \1_{q \in \{k, \ldots, k+n\}}
\right)_{q=1, \ldots, k+n}
,
\
\left(
(-i \omega_j)^k d_j \1_{j \in \N^*}
\right)_{j \in \N}
\right),
\end{equation*}
and
\begin{equation}
\label{op_primitive}
\mathcal{B}_k(v)
:=
\left( t \mapsto \int_0^t \frac{(t-\tau)^{k-1}}{(k-1)!} v(\tau) d\tau \right).
\end{equation}
\end{proof}

\subsection{Solvability of a moment problem in $H^k_0(0,T)$ with various estimates}
Notice that \cref{thm_mom_Hm} provides operators $\tild{\L}_k^T :  \left( (d_{-q})_{q=1, \ldots, n}, \ (d_j)_{j \in \mathbb{N}} \right)  \in \R^n \times h^{2k}_{\omega,r}( \mathbb{N} , \mathbb{C}) \mapsto u \in H^k_0( (0,T), \R)$ solving the moment problem which depend on $k$, preventing from having estimates on $u$, for a given sequence $(d_j)$, simultaneously in various Sobolev spaces. Therefore, the goal of this subsection is to prove that one can solve trigonometric moment problems in $H^k_0(0,T)$ with simultaneous estimates on the control.
%
%
%
%
%

First, the result can be proved when dealing with only a finite number of moments. 
\begin{prop}
\label{thm_mom_fini}
Let $T>0$, $(n, k, N) \in \N^* \times \N^2$ and $(\omega_j)_{j \in \mathbb{N}}$ an increasing sequence of $[0, +\infty)$ such that $\omega_0=0$ and satisfying \eqref{AsymptGap}. There exists a constant $C_N>0$ and a continuous linear map $\L^{N, T}_{\lf} : \R^n \times \left( \R \times \C^{N-1} \right) \rightarrow H^k_0((0,T), \R)$ such that for every sequence $d=\left( (d_{-q})_{q=1, \ldots, n}, \ (d_j)_{j=0, \ldots, N-1} \right) \in \R^n \times \left( \R \times \C^{N-1} \right)$, the control $u:=\L^{N, T}_{\lf}(d) \in H^k_0( (0,T), \R)$ satisfies the moment problem
\begin{equation*}
\forall  j=0, \ldots, N-1, \ \int_0^T u(t) e^{i \omega_j t}dt = d_j, \quad \text{ and } \quad \forall  j \soe N, \ \int_0^T u(t) e^{i \omega_j t}dt =0,
\end{equation*}
\begin{equation*}
\forall q=1, \ldots, n, \ \int_0^T t^q u(t) dt = d_{-q},
\end{equation*}
and the size estimates
\begin{equation}
\label{estim_moment_finie}
\| u \|_{H^m_0(0,T)} \ioe C_N \left( \sum \limits_{j=-n}^{N-1} \left| \left( \delta_{j,0}+\omega_j^{m} \right) d_j \right|^2\right)^{1/2}, \quad  \forall m=0, \ldots, k. 
\end{equation}
\end{prop}

\begin{proof}
The proof follows with 
\begin{equation*}
\L^{N, T}_{\lf}\left( (d_{-q})_{q=1, \ldots, n}, \ (d_j)_{j=0, \ldots, N-1} \right):=\tild{\L}^T_{k}\left( (d_{-q})_{q=1, \ldots, n}, \ (d_j \1_{j=0, \ldots, N-1})_{j \in \N} \right)
\end{equation*}
where $\tild{\L}^T_k$ is defined in \cref{thm_mom_Hm}, using the equivalence of norms in finite dimension.
\end{proof}



It remains to deal with the high frequencies. To that end, we will assume from now on that the sequence of frequencies satisfies the polynomial asymptotic gap \eqref{AsymptGapPoly}. In other words, the goal is to prove that the map 
$$L : u \mapsto  \left( \int_0^T u(t) e^{i \omega_j t} dt \right)_{j \in \N}$$
admits a continuous-right inverse $P: h^{2k}_{\omega,r}(\N, \C) \rightarrow H^k_0((0,T), \R)$ which is still continuous from $ h^{2m}_{\omega,r}(\N,\C)$ to $H^m_0((0,T), \R)$ for all $m=0, \ldots, k$. Usually, a continuous right inverse of $L$ i.e.\ an operator solving the moment problem is sought under the form
$P(d)= \sum  d_j \xi_j^*$
where $\{ \xi_j^*, j \in \Z \}$ is the biorthogonal family of $\{ e^{i \omega_j \cdot}, j \in \Z \}$. To conclude, one would need to be able to estimate such biorthogonal family simultaneously in all the Sobolev spaces $H^m_0(0,T)$, for $m=0, \ldots, k$. 

Such strategies have already been used. Explicit computations of the biorthogonal family with good estimates have, for example, been used: in \cite{TT07} to prove upper bounds for the control cost in the case of systems governed by the Schrödinger or the heat equation, 
in \cite{L17} to study the cost of the control in the case of a minimal time for the one-dimensional heat equation or in \cite{BBGBO14} to characterize the null controllability of a system of $n$ parabolic equations in cylindrical domains. Sharp estimates for biorthogonal families of exponential functions without gap conditions have been given in \cite{GBO21} and used to prove new results on the cost of the boundary null controllability of parabolic systems. A new block resolution technique, together with sharp estimates, has also been used in \cite{BBM20} to characterize the minimal null control time for abstract linear control problem. 

However, here we choose to not compute the biorthogonal family. As the exponentials are "almost orthogonal" for high frequencies, the main idea is to rather seek a solution of the moment problem under the form
\begin{equation}
\label{form_P_good}
P(d)(t)= \sum \limits_{j=-\infty}^{+\infty} d_j e^{i \omega_jt} \chi(t),
\end{equation}
with $\chi$ a weight function which allows to improve the decay of the coefficients $\left( \langle e^{i \omega_j \cdot} \chi,e^{i \omega_p \cdot} \rangle\right)_{j \neq p}$ at high frequencies. As such $P$ will no longer exactly solve the moment problem, the right inverse of $L$ will be constructed as an iteration of \eqref{form_P_good}, quantifying the error term. Besides, the explicit form of $P$ will allow to easily estimate it in various spaces. 

To implement such strategy, we start by introducing the operator giving the moment problem for high frequencies and the operator which will almost be its right-inverse. 

\begin{lem}
Define, for all $m \in \N$ and $N \in \N^*$, 
$$
\begin{array}{|lrcl}
L_N^m : & H^m_0((0,T), \R) & \rightarrow & h^{2m}_{\omega,r}(\N, \C) \\
    & u & \mapsto & \left( \int_0^T u(t) e^{i \omega_j t} dt \right)_{j \geqslant N},
     \end{array}
\begin{array}{|lrcl}
P_N^m : & h^{2m}_{\omega,r}(\N, \C) & \rightarrow & H^m_0((0,T),\R)  \\
    & (d_j)_{j\geqslant N} & \mapsto & \sum \limits_{| j | \geq N} d_j \widetilde{\xi_j}, \end{array}
$$ 
where for all $j \in \Z,$ $j<0$, $d_j := \overline{d_{-j}}$ and for all $j \in \Z$, for all $t \in [0,T]$ $\widetilde{\xi_j}(t) := \frac{1}{T} e^{-i \omega_j t} \chi(t)$ with $\chi \in C^{\infty}_c((0,T), \R)$ such that $\int_0^T \chi(t) dt=1.$
Then, for all $m \in \N$ and $N \in \N^*$, $L_N^m$ and $P_N^m$ are linear continuous applications. 
\end{lem}

\begin{proof}
Let $m \in \N$ and $N \in \N^*$.
%
%
%
First, the continuity of $L^m_N$ comes from that for all $u \in H^m_0(0,T)$, by integrations by parts,
\begin{equation*}
\sum \limits_{j=N}^{+\infty} 
\left|
\omega_j^m
\int_0^T u(t) e^{i \omega_j t}
dt
\right|^2
= 
\sum \limits_{j=N}^{+\infty} 
\left|
\langle u^{(m)}, e^{i \omega \cdot} \rangle
\right|^2
\ioe C \| u^{(m)} \|_{L^2(0,T)}^2,
\end{equation*}
as the family $(e^{i \omega_j \cdot})_{j \in \Z}$ is a Riesz basis. Secondly, the continuity of $P^m_N$ stems from the fact that for all $(d_j)_{j \geq N}$ in $h^{2m}_{\omega,r}(\N, \C)$, by the algebra structure of $H^m_0(0,T)$, 
 \begin{equation*}
 \left\| \sum \limits_{| j | \soe N} d_j \tild{\xi_j} \right\|_{H^m_0(0,T)} 
\ioe  
\| \chi \|_{H^m(0,T)} 
\left\| 
\sum \limits_{| j | \soe N} \omega_j^m d_j e^{-i \omega_j \cdot} 
\right\|_{L^2(0,T)} 
\ioe 
C
\| \left( \omega_j^m d_j \right)_{ | j | \soe N} \|_{l^2(\N, \C)},
 \end{equation*}
as the family $(e^{i \omega_j \cdot})_{j \in \Z}$ is a Riesz basis. The reader can for example refer to \cite[Proposition 19]{BL10} to find the results on Riesz basis used in this proof. 
%
%
%
\end{proof}
With these notations, our goal is to prove that, for $N$ large enough, the application $L_N^m$ has a common continuous right inverse for all $m=0, \ldots, k$. To that end, we start by quantifying in which way $P_N^m$ is almost the right-inverse of $L_N^m$. 
\begin{lem}
\label{almost_inv}
Let $k \in \N^*$. For all $\eps>0$, there exists $N_1 \in \N^*$ such that for all $N \soe N_1$, for all $m=0, \ldots, k$, for all $d \in h^{2m}_{\omega, r}(\N, \C)$, 
\begin{equation*}
\| L_N^m \circ P_N^m(d) - d \|_{h^{2m}_{\omega,r}(\N, \C)} \ioe \eps \| d\|_{h^{2m}_{\omega,r}(\N, \C)}.
\end{equation*}
\end{lem}
\begin{proof}
Let $m \in \{0, \ldots, k\}$ and $N \soe \min(N_0+1, 2)$ where $N_0$ is defined in \eqref{AsymptGapPoly}. Notice that, performing integrations by parts (with no boundary terms as $\chi$ has a compact support), for all $\alpha \in \N^*$, there exists a constant $C=C( \chi^{(\alpha)}, T)>0$ such that, for all $j, p \in \N$, 
\begin{equation*}
\int_0^T \widetilde{\xi_j}(t)e^{i\omega_jt} dt=1 
\quad \text{ and } \quad
\left|\int_0^T \widetilde{\xi_j}(t)e^{i\omega_kt} dt\right| \ioe \frac{C}{ \left|\omega_k-\omega_j\right|^{\alpha}} \text{ if } j \neq k.
\end{equation*}
The coefficient $\alpha$ will be chosen later as large as needed. Using this remark, together with the following equality
\begin{equation*}
\left( L_N^m \circ P_N^m(d) \right)_p
= \int_0^T P_N^m(d)(t) e^{i \omega_pt } dt 
= d_p + \sum \limits_{|j | \soe N, \ j \neq p} d_j \int_0^T  \widetilde{\xi_j}(t) e^{i\omega_p t} dt, 
\quad \forall p \in \N,
\end{equation*} 
we get, using Cauchy-Schwarz inequality, 
\begin{align*}
\| L_N^m \circ P_N^m -d \|_{h^{2m}_{\omega,r}}^2 
&= \sum \limits_{p \soe N} \omega_p^{2m} \Big| \sum \limits_{|j | \soe N, \ j \neq p} d_j \int_0^T  \widetilde{\xi_j}(t) e^{i\omega_p t} dt \Big|^2 \\
&\ioe C \|d \|^2_{h^{2m}_{\omega, r}} \sum \limits_{p \soe N}  \sum_{j  \soe N \atop j \neq p} \frac{\omega_p^{2m}}{\omega_j^{2m}| \omega_j -\omega_p |^{2\alpha}}. 
\end{align*}
Yet, by the triangular inequality, for all $(j, p) \in \N^*$, $j \neq p$, 
\begin{equation*}
\frac{\omega_p^{2m}}{\omega_j^{2m}| \omega_j -\omega_p |^{2\alpha}}
\ioe
C
\left(
\frac{1}{ \omega_j^{2m} | \omega_j - \omega_p|^{2(\alpha-m)}} + \frac{1}{ | \omega_j - \omega_p|^{2\alpha}}
\right)
.
\end{equation*}
Thus, as the sequence $(\omega_j)_{j \in \N^*}$ is bounded by below by $\omega_1$, it is sufficient to prove that for $\beta$ large enough, the series $\sum_{j,p} \frac{1}{ | \omega_j - \omega_p|^{\beta}}$ converges to get that choosing $\alpha$ large enough, for all $\eps>0$, there exists $N_1>0$ such that for all $N \soe N_1$, for all $m=0, \ldots, k$, 
$$
 \sum \limits_{p \soe N}  \sum_{j  \soe N \atop j \neq p} \frac{\omega_p^{2m}}{\omega_j^{2m}| \omega_j -\omega_p |^{2\alpha}} \ioe \eps,
$$
which will conclude the proof. Yet, using the polynomial growth \eqref{AsymptGapPoly},  for all $j, p \soe N$, $j \neq p$, 
\begin{equation*}
| \omega_j- \omega_p| 
\soe 
\min
\left(
\sqrt{
\left(
\omega_j - \omega_{j-1}
\right) 
\left(  
\omega_{p+1} - \omega_{p} 
\right)
}
, 
\sqrt{ 
\left(
\omega_{j+1} - \omega_{j}
\right)
\left(
\omega_{p} - \omega_{p-1}
\right)
}
\right)
\soe 
c (j-1)^{\eps/2} (p-1)^{\eps/2}. 
\end{equation*}
And thus, for all $\beta > \frac{2}{\eps}$, the series $\sum_{j,p} \frac{1}{ | \omega_j - \omega_p|^{\beta}}$ indeed converges.
\end{proof}
By iterating the error estimate given in \cref{almost_inv}, we deduce a common continuous right-inverse for $L^m_N$, the operator of the moment problem.
\begin{lem}
\label{inv}
Let $k \in \N^*$. There exists $N_1 \in \N^*$ such that for all $N \soe N_1$, $L_N^m$ admits the same linear continuous right inverse for all $m=0, \ldots, k$. More precisely, for all $N \soe N_1$, there exists an operator $M_N$ such that for all $m=0, \ldots, k$, $M_N^m : h^{2m}_{\omega, r}(\N, \C) \rightarrow H^m_0((0,T),\R), d \rightarrow M_N(d)$ is continuous and satisfies
$$
L_N^m \circ M_N^m = \Id_{h^{2m}}.
$$
\end{lem}

\begin{proof}
By \cref{almost_inv} with $\eps=\frac{1}{2}$, set $N_1>0$ such that for all $N \soe N_1$, for all $m=0, \ldots, k$, we have, 
\begin{equation}
\label{estim_LN_PN}
\| L_N^m \circ P_N^m(d) - d \|_{h^{2m}_{\omega, r}(\N,\C)} \ioe \frac{1}{2} \| d\|_{h^{2m}_{\omega,r}(\N,\C)}, \quad \forall d \in h^{2m}_{\omega, r}(\N,\C).
\end{equation}
Let $N \soe N_1$ and $m \in \{0, \ldots, k\}$. In the following proof, for the sake of clarity, we will forget all the exponents $m$ on the name of the applications, that recall the spaces in which we work. First, if we define $\Sigma:=L_N \circ P_N - \Id$, an induction gives that 
\begin{equation}
\label{hyp_rec}
L_N \circ \sum \limits_{p=0}^n (-1)^p P_N \circ \Sigma^p = \Id + (-1)^n \Sigma^{n+1}, \quad \forall n \in \N.
\end{equation}
Yet, $\Sigma$ is a linear and continuous application from $h^{2m}_{\omega,r}(\N,\C)$ to $h^{2m}_{\omega, r}(\N,\C)$ with its operator norm satisfying $\| \Sigma \|_{h^{2m},h^{2m}} \ioe 1/2$ by \eqref{estim_LN_PN}. Thus, the series $\sum_p  (-1)^p P_N \circ \Sigma^p $ absolutely converges in the space $\mathcal{L}_c(h^{2m}_{\omega, r}(\N, \C), H^m_0((0,T),\R))$. Therefore, passing to the limits [$n \rightarrow +\infty$] in the equality \eqref{hyp_rec}, we get, 
\begin{equation*}
L_N \circ M_N=\Id \quad \text{ with } \quad M_N:=\sum \limits_{p=0}^{+\infty}  (-1)^p P_N \circ \Sigma^p,
\end{equation*}
which is continuous from $h^{2m}_{\omega, r}(\N, \C)$ to $H^m_0((0,T),\R)$ for all $m=0, \ldots, k$.
\end{proof}
Now, from \cref{inv}, one can prove the solvability of a moment problem with simultaneous estimates on the control for high frequencies. 
\begin{prop}
\label{thm_mom_hf}
Let $T>0$, $k \in \N$ and $(\omega_j)_{j \in \N}$ an increasing sequence of $[0, +\infty)$ such that $\omega_0=0$ and satisfying \eqref{AsymptGapPoly}. There exists an integer $N \in \N^*$, a constant $C>0$ and a continuous linear map $\L^{N, T}_{\hf} : h^{2k}_{\omega,r}(\N,\C) \rightarrow H^k_0((0,T), \R)$ such that for every sequence $d= (d_j)_{j \in \N}\in h^{2k}_{\omega,r}(\N,\C) $, the control $u:=\L^{N, T}_{\hf}(d) \in H^k_0( (0,T), \R)$ satisfies the moment problem  
\begin{equation}
\label{mom_hf}
\int_0^T u(t) e^{i \omega_j t} dt = d_j, \quad \forall j \geq N,
\end{equation}
and the following estimates
\begin{equation}
\label{estim_hf}
\| u \|_{H^m_0(0,T)} \ioe C  \left( \sum \limits_{j=N}^{+\infty} \left| \omega_j^{m} d_j \right|^2\right)^{1/2}, \quad  \forall m=0, \ldots, k. 
\end{equation}
\end{prop}
\begin{proof}[Proof of \cref{thm_mom_hf}.]
Let $k \in \N$ and $T>0$. Let $N_1 \in \N$ as in \cref{inv} and $N \soe N_1$. The proof follows with $\L^{N, T}_{\hf}:=M_N$ where $M_N$ is defined in \cref{inv}.
Indeed, for all $d \in h_{\omega, r}^{2k}(\N, \C)$, as $L_N \circ M_N = \Id$, the function $u:=M_N(d) \in H^k_0( (0,T), \R)$ satisfies the moment problem \eqref{mom_hf}. And the estimates \eqref{estim_hf} hold by continuity of $M_N$ from $h^{2m}_{\omega, r}(\N,\C)$ to $H^m_0((0,T),\R)$, for every $m=0, \ldots, k$.
\end{proof}
Now, from \cref{thm_mom_fini} dealing with low frequencies and \cref{thm_mom_hf} dealing with high frequencies, one can prove the main result. 
\begin{thm}
\label{thm_mom}
Let $T>0$, $k \in \N$, $n \in \N^*$ and $(\omega_j)_{j \in \N}$ an increasing sequence of $[0, +\infty)$ such that $\omega_0=0$ and satisfying \eqref{AsymptGapPoly}. There exists a constant $C>0$ and a continuous linear map $\L^T_k : \R^n \times h^{2k}_{\omega,r}(\N, \C) \rightarrow H^k_0((0,T), \R)$ such that for every sequence $d=\left( (d_{-q})_{q=1, \ldots, n}, \ (d_j)_{j \in \mathbb{N}} \right) \in \R^n \times h^{2k}_{\omega,r}( \mathbb{N} , \mathbb{C})$, the control $u:=\L^T_k(d) \in H^k_0( (0,T), \R)$ satisfies the moment problem 
\begin{equation}
\label{moment}
\forall j \in \mathbb{N}, \int_0^T u(t) e^{i \omega_j t}dt = d_j
\quad
\text{ and }
\quad
\forall q=1, \ldots, n, \int_0^T t^q u(t) dt= d_{-q},
\end{equation}
with the following size estimates,
\begin{equation}
\label{estim_moment}
\| u \|_{H^m_0(0,T)} \ioe C \left(
\sum \limits_{j=-n}^{+\infty} \left| \left( \delta_{j,0} + \omega_j^m \right) d_j \right|^2
\right)^{1/2}, \quad  \forall  m=0, \ldots, k. 
\end{equation}
\end{thm}

\begin{proof}[Proof of \cref{thm_mom}.]
Let $T>0$, $k \in \N$, $n \in \N^*$. Let $(d_{-q})_{q=1, \ldots,n} \in \R^n$ and $(d_j)_{j \in \mathbb{N}} \in h^{2k}_{\omega, r}( \mathbb{N} , \mathbb{C})$.
The proof follows with
\begin{multline*}
u
=
\L^T_k
\left( 
(d_{-q})_{q=1, \ldots,n}
,
(d_j)_{j \in \N} 
\right)
:=
\L^{N, T}_{\hf} 
( (d_j)_{j \in \N} )
\\+
\L^{N, T}_{\lf} 
\left(
\left(
d_{-q}
-
\int_0^T 
t^q
\L^{N, T}_{\hf} 
( d_j )(t)
dt
\right)_{q=1, \ldots, ,n}
,
\left(
d_{j}
-
\int_0^T 
\L^{N, T}_{\hf} 
( d_j )(t)
e^{i \omega_j t}
dt
\right)_{j=0, \ldots, N-1}
\right).
\end{multline*}
Indeed, by linearity and by construction of the operators $\L_{\hf}^{N, T}$ and $\L_{\lf}^{N, T}$ (given respectively in \cref{thm_mom_hf} and \cref{thm_mom_fini}, the control $u$ satisfies the moment problem \eqref{moment}. Furthermore, by \eqref{estim_moment_finie} and \eqref{estim_hf}, there exists a constant $C>0$ such that for all $m=0, \ldots, k$, 
\begin{equation}
\label{eq2}
\| u \|_{H^m_0(0,T)} 
\ioe 
C 
\left( 
\left( 
\sum \limits_{j=N}^{+\infty} 
\left|
\omega_j^m 
d_j  
\right|^2
\right)^{1/2} 
+ 
\left( 
\sum \limits_{j=-n}^{N-1} 
\left| 
\left( 
\delta_{j,0}+{\omega_j}^{m} 
\right) 
\left(
\widetilde{d}_j -d_j
\right)
\right|^2 
\right)^{1/2} 
\right). 
\end{equation}
where, if $v:=\L^{N, T}_{\hf} 
( (d_j)_{j \in \N} )$,
\begin{equation*}
\forall j=0, \ldots, N-1, \ \tild{d}_j:= \int_0^T v(t) e^{i \omega_j t} dt
\quad
\text{ and }
\quad
\forall q=1, \ldots, n, \ \widetilde{d}_{-q}:=\int_0^T t^q v(t)dt.
\end{equation*}
Besides, using Cauchy-Schwarz inequality and the size estimate \eqref{estim_hf} on $v$ (for $m=0$), we get, for all $j=0, \ldots , N-1$, for all $m=0, \ldots, k,$
\begin{equation*}
| \widetilde{d}_j |
= \left| \int_0^T v(t) e^{i \omega_jt} dt \right| \\
\ioe \sqrt{T} \| v \|_{L^2(0,T)} \\
\ioe C  \left( \sum \limits_{j=N}^{+\infty} \left| d_j \right|^2 \right)^{1/2} \\
\ioe C  \left( \sum \limits_{j=-n}^{+\infty} \left| \left( \delta_{j,0}+{\omega_j}^{m} \right)  d_j  \right|^2 \right)^{1/2},
\end{equation*}
as the injection $h^{2m}_{\omega,r}(\N,\C) \subset l^2_r(\N,\C)$ is continuous. The same estimates can be proved on $(\tild{d}_{-q})_{q=1, \ldots, n}$. Together with \eqref{eq2}, this gives \eqref{estim_moment}. 
\end{proof}
%
Finally, from \cref{thm_mom},  one can deduce the solvability of moment problem with estimates on the function, some of its derivatives but also some of its primitives. 
\begin{thm}
\label{thm_mom_weak_estim}
Let $T>0$, $k \in \mathbb{N}^*$ and $(\omega_j)_{j \in \mathbb{N}}$ an increasing sequence of $[0, +\infty)$ such that $\omega_0=0$ and satisfying \eqref{AsymptGapPoly}. There exists a constant $C>0$ and a continuous linear map $\mathcal{M}^T_k : h^{2k}_{\omega, r}(\N,\C) \rightarrow H^k_0( (0,T), \R)$ such that for every $d=(d_j)_{j \in \mathbb{N}} \in h^{2k}_{\omega, r}( \mathbb{N}, \mathbb{C})$, the control $u:= \mathcal{M}^T_k(d) \in H^k_0( (0,T), \R)$ satisfies the moment problem
\begin{equation}
\label{mom_u}
\int_0^T u(t) e^{i \omega_j t} dt = d_j, \quad \forall j \in \mathbb{N}, 
\end{equation} 
with the boundary conditions 
\begin{equation}
\label{weak_bc_u}
u_2(T)=u_3(T)=\ldots=u_{k+1}(T)=0,
\end{equation} 
and the following size estimates,
\begin{equation}
\label{estim_u}
\| u \|_{H^m_0(0,T)} \ioe C \| d \|_{h^{2m}_{\omega}(\N,\C)}, \quad \forall m=-(k+1), \ldots, 0, \ldots, k. 
\end{equation}
\end{thm}
\begin{proof}[Proof of \cref{thm_mom_weak_estim}]
Let $(d_j)_{j \in \N}$ in $h^{2k}_{\omega, r}(\N,\C)$. The operator $\mathcal{M}_k^T$ is given by, 
\begin{equation*}
\mathcal{M}_k^T
\left( 
(d_j)_{j \in \N}
\right)
:=
D^{(k+1)} \circ 
\L_{2k+1}^T 
\left(
\frac{d_j}{ (-i\omega_j)^{k+1}} \1_{j \in \N^*}
\right)_{j \in \N}
+
\L_{k}^T
\left(
\left(
T^q d_0
\right)_{q=1, \ldots, k}
,
\left(
d_0 \delta_{j,0}
\right)_{j \in \N}
\right),
\end{equation*}
where $\L_{2k+1}^T$ and $\L_k^T$ are defined in \cref{thm_mom} and $D^{ (k+1)}$ is the differential operator $ u \mapsto u^{(k+1)}$.  

\medskip \noindent \emph{Step 1. Solving the moment problem except for the first moment.}
Indeed, if we denote by 
\begin{equation*}
v:=f^{(k+1)} \in H^k_0(0,T) \quad \text{ with } \quad 
f:= 
\L_{2k+1}^T 
\left(
\frac{d_j}{ (-i\omega_j)^{k+1}} \1_{j \in \N^*}
\right)_{j \in \N}
\in 
H^{2k+1}_0(0,T)
, 
\end{equation*}
then, performing $k+1$ integrations by parts, as $f^{(m)}$ has vanishing boundary terms for all $m=0, \ldots, k$, we get, for all $j \in \N^*$, 
\begin{equation}
\label{mom_v}
\int_0^T v(t) e^{i \omega_j t} dt = (-i \omega_j)^{k+1} \int_0^T f(t) e^{i \omega_j t} dt=d_j,
\end{equation}
by construction of $f$. However, for $j=1$, we get 
\begin{equation}
\label{mom_0_v}
\int_0^T v(t) dt = f^{(k)}(T)=0,
\end{equation}
and therefore, the first moment needs to be corrected in a second time. Moreover, by construction, we have the boundary conditions \eqref{weak_bc_u} on $v$
and from estimates \eqref{estim_moment} on $f$, if we denote by $u^{(m)}=u_{-m}$ when $m <0$, we deduce
\begin{align}
\label{estim_strong_v}
\| v \|_{H^m} &= \| v^{(m)}\|_{L^2} = \| f^{(k+1+m)} \|_{L^2} \ioe C \| c\|_{h^{2(k+1+m)}_{\omega,r}} \ioe C \| d\|_{h^{2m}_{\omega,r}}, \quad \forall m=-(k+1), \ldots, k. 
\end{align}

\medskip \noindent \emph{Step 2. Correcting the first component.}
If we denote by 
\begin{equation*}
w:=\L_{k}^T
\left(
\left(
T^q d_0
\right)_{q=1, \ldots, k}
,
\left(
d_0 \delta_{j,0}
\right)_{j \in \N}
\right)
\in 
H_0^k(0,T),
\end{equation*}
then by construction, together with \eqref{mom_v}-\eqref{mom_0_v}, the control $u:=v+w$ in $H^k_0(0,T)$ solves the moment problem \eqref{mom_u}. Moreover, by construction, $w$ solves the polynomial moment 
\begin{equation*}
\int_0^T t^q w(t) dt = T^q d_0, \quad \forall q=1, \ldots, k 
\end{equation*} 
and thus, by integration by parts, as $w_1(T)=d_0$ by construction, we get the boundary conditions \eqref{weak_bc_u} on $w$. Then, by linearity, \eqref{weak_bc_u} holds for $u$.
Moreover, by construction, $w$ satisfies the size estimates
\begin{equation}
\label{estim_strong_w}
\| w \|_{H^m(0,T)} \ioe C_T | d_0|, \quad \forall m=0, \ldots, k. 
\end{equation}
Thus, Cauchy-Schwarz inequality entails that
\begin{equation}
\label{estim_weak_w}
\| w_m\|_{L^2(0,T)} \ioe T^m \| w \|_{L^2(0,T)} \ioe T^m C_T | d_0|, \quad \forall m=1, \ldots, k+1.
\end{equation}
Finally, estimates \eqref{estim_strong_v} on $v$ and estimates \eqref{estim_strong_w}-\eqref{estim_weak_w} on $w$ gives all the estimates \eqref{estim_u} on $u$. 
\end{proof}
%
%
%
%
%
\section{Nonlinear control in projection with simultaneous estimates}
\label{main_theorem}
The goal of this section is the proof of \cref{thm:contr_lin_proj}. It relies on the controllability of the linearized system with simultaneous estimates, given in Subsection \ref{C1-regularity}, which is then propagated to the nonlinear system through the iterations of the inverse mapping theorem thanks to estimates on the linear approximation of the end-point map given in Subsection \ref{linear_approx}. In this section, if $\psi$ is in $\S$, $T_{\S} \psi$ the tangent space at $\psi$ of $\S$ and $\Pi_{\psi}$ the orthogonal projection on $T_{\S} \psi$, are respectively given by
\begin{equation*}
T_{\S} \psi := \{ \xi \in L^2( 0,1) ; \ \Re \langle \xi, \psi \rangle=0 \} \quad \text{ and } \quad \Pi_{\psi}( \xi):=\xi - \Re \langle \xi, \psi \rangle \psi.
\end{equation*}


\subsection{$C^1-$ regularity of the end-point map}
\label{C1-regularity}
Let $T>0$. We consider the end-point map defined by,
 \begin{equation}
 \label{def_theta_T_proj}
 \begin{array}{ccrcl}
\Theta_T & : & \S \cap H^{2(p+k)+3}_{(0)} \times H^k_0 & \to & \S \cap H^{2(p+k)+3}_{(0)} \times \left[ T_{\S} \psi_1(T) \cap \H \cap H^{2(p+k)+3}_{(0)} \right] \\
 & & (\psi_0, \ u) & \mapsto & \left(\psi_0, \  \Pi_{\psi_1(T)} \circ \P_{\J} \left[ \psi(T) \right] \right) \\
\end{array}
\end{equation}
where $\psi$ is the solution of
\begin{equation*}  
\left\{
    \begin{array}{ll}
        i \partial_t \psi(t,x) = - \partial^2_x \psi(t,x) -u(t)\mu(x)\psi(t,x),  \quad (t,x) \in (0,T) \times (0,1),\\
        \psi(t,0) = \psi(t,1)=0, \quad t \in (0,T), \\
        \psi(0,x) = \psi_0(x), \quad x \in (0,1).
    \end{array}
\right.
\end{equation*}
\begin{rem}
To prove \cref{thm:contr_lin_proj}, we want estimates \eqref{estim_contr_nl} on the control with respect to the final state but also to the initial data, that is why the initial data is added as an argument of the end-point map $\Theta_T$. Moreover, in the definition of $\Theta_T$, the end-point of the solution is composed with $\P_{\J}$ as we investigate exact controllability in projection. It is also composed with $\Pi_{\psi_1(T)}$ as only the imaginary part of the first component of the solution can be controlled (we gain back the real part of the first component as the solution lives in the $L^2(0,1)$-sphere $\S$).
\end{rem}
%
The $C^1$-regularity of this end-point map is given in the following proposition.
\begin{prop}
\label{diff_eq_proj}
Let $\mu \in H^{2(p+k)+3}( (0,1), \R)$ with $\mu^{(2n+1)}(0)=\mu^{(2n+1)}(1)=0$ for all $n=0, \ldots, p-1$. The map $\Theta_T$ defined  in \eqref{def_theta_T_proj} is $C^1$. Moreover, for every $(\psi_0, u)$ in $H^{2(p+k)+3}_{(0)}(0,1) \times H^k_0(0,T)$, the differential at $(\psi_0,u)$ is given by, 
\begin{equation*}
d \Theta_T (\psi_0,u) . ( \Psi_0, v)= \left(\Psi_0, \Pi_{\psi_1(T)} \circ \P_{\J} \left[ \Psi(T) \right] \right), \quad (\Psi_0,v) \in H^{2(p+k)+3}_{(0)}(0,1) \times H^k_0(0,T),
\end{equation*}
 where $\Psi$ is the solution of the linearized system around the trajectory $\left(u, \psi( \cdot; \ u, \psi_0)\right)$ given by
 \begin{equation}
 \label{diff_u_psi0_T0}
 \left\{
    \begin{array}{ll}
        i \partial_t \Psi(t,x) = - \partial^2_x \Psi(t,x) -u(t) \mu(x) \Psi(t,x)-v(t) \mu(x) \psi(t; \ u, \psi_0),\\
        \Psi(t,0) = \Psi(t,1)=0, \quad t \in (0,T), \\
        \Psi(0,x)=\Psi_0, \quad x \in (0,1),
    \end{array}
\right.  \end{equation}
\end{prop}
Notice that $\Theta_T$ is well-defined thanks to \cref{rem:point_final}. The proof of \cref{diff_eq_proj} is the same as \cite[Proposition 3, Proposition 6]{BL10} using \cref{thm:well-posedness} (and \cref{rem:point_final}) instead of their well-posedness result and thus is left to the reader.  

In the following proposition, we state that one can build a right inverse of $d \Theta_T(\varphi_1, 0)$ which is continuous simultaneously in several spaces, meaning that one can control the linearized equation of \eqref{Schrodinger} around the ground state with various estimates on the control. 
\begin{prop}
\label{inv_droite_con}
Let $\mu$ in $H^{2(p+k)+3}( (0,1), \R)$  with $\mu^{(2n+1)}(0)=\mu^{(2n+1)}(1)=0$ for all $n=0, \ldots, p-1$ and satisfying \eqref{hyp_mu} . Then, the linear map 
\begin{equation*}
d \Theta_T(\varphi_1, 0) : 
[T_{\S} \varphi_1 \cap H^{2(p+k)+3}_{(0)}] \times H^k_0 
  \rightarrow 
[T_{\S} \varphi_1 \cap H^{2(p+k)+3}_{(0)}] \times [T_{\S} \psi_1(T) \cap \H \cap H^{2(p+k)+3}_{(0)}]
\end{equation*}
has a continuous right inverse
\begin{equation*}
d \Theta_T (\varphi_1, 0)^{-1} :  
[T_{\S} \varphi_1  \cap H^{2(p+k)+3}_{(0)}] \times [T_{\S} \psi_1(T) \cap \H \cap H^{2(p+k)+3}_{(0)}]
\rightarrow  
[T_{\S} \varphi_1  \cap H^{2(p+k)+3}_{(0)}] \times H^k_0,
\end{equation*}
which satisfies that there exists a constant $C_T>0$ such that for all $\psi_0$ in $T_{\S} \varphi_1 \cap H^{2(p+k)+3}_{(0)}$ and $\psi_f$ in $T_{\S} \psi_1(T) \cap \H \cap H^{2(p+k)+3}_{(0)}$, the control $u \in H_0^k(0,T)$ defined by $(\psi_0, u):= d\Theta_T(\varphi_1, 0)^{-1} (\psi_0, \psi_f)$ satisfies the following boundary conditions
\begin{equation}
\label{bc}
u_2(T)=\ldots=u_{k+1}(T)=0
\end{equation}
and the following size estimates
\begin{equation}
\label{size_estimate}
\left\| u \right\|_{H^{m}_0(0,T)} 
\ioe 
C 
\left\| (\psi_0, \psi_f )\right\|_{H^{2(p+m)+3}_{(0)}(0,1) \times H^{2(p+m)+3}_{(0)}(0,1)}, \quad \forall m=-(k+1),\ldots,k.
\end{equation}
 \end{prop}
The proof follows from the solvability of a trigonometric moment problem with simultaneous estimates given in \cref{thm_mom_weak_estim}.
\begin{proof}[Proof of \cref{inv_droite_con}]
By \cref{diff_eq_proj}, for all $u \in H_0^k(0,T)$ and $\psi_0 \in H^{2(p+k)+3}_{(0)}(0,1)$, 
$$d \Theta_T (\varphi_1,0). ( \psi_0, u)= \left(\psi_0, \P_{\J} \left[ \Psi(T) \right] \right),$$
where $\Psi$ is the solution of the linearized system \eqref{diff_u_psi0_T0} around $(u, \psi_1)$ and can be computed as
\begin{equation}
\label{expr_lin}
\Psi(T) 
=
\sum \limits_{j=1}^{+\infty} 
\langle \psi_0, \varphi_j \rangle 
\ \psi_j(T)
+ 
i
\sum \limits_{j=1}^{+\infty} 
\left(
\langle \mu \varphi_1, \varphi_j \rangle 
\int_{0}^T u(t) e^{i (\lambda_j - \lambda_1) t} dt 
\right)
\psi_j(T). 
\end{equation}
Let $(\psi_0, \psi_f) \in [T_{\S} \varphi_1 \cap H^{2(p+k)+3}_{(0)}] \times [T_{\S} \psi_1(T) \cap \H \cap H^{2(p+k)+3}_{(0)}]$. The equality $\P_{\J} \Psi(T)=\psi_f$ is then equivalent to the trigonometric moment problem 
\begin{equation}
\label{pb_mom_v}
\int_0^T u(t) e^{i (\lambda_j - \lambda_1) t} dt = \frac{ \langle \psi_f , \psi_j(T) \rangle- \langle \psi_0, \varphi_j \rangle}{i \langle \mu \varphi_1, \varphi_j \rangle}:=d_{j-1}(\psi_0, \psi_f), \quad \forall j \in J.  
\end{equation}
Applying \cref{thm_mom_weak_estim} with $(\omega_j:=\lambda_{j+1}-\lambda_1)_{j \in \N}$  which satisfies \eqref{AsymptGapPoly}, the proof of \cref{inv_droite_con} follows with
\begin{equation}
\label{def_right_inv}
d \Theta_T(\varphi_1,0)^{-1} (\psi_0, \psi_f):= \left( \psi_0, \mathcal{M}^T_k \left[ d(\psi_0, \psi_f) \right] \right),
\end{equation}
where $d(\psi_0, \psi_f) := \left( d_j(\psi_0, \psi_f) \mathbb{1}_{j \in J} \right)_{j \in \N}$ and $\mathcal{M}^T_k$ is defined in \cref{thm_mom_weak_estim}. 
Indeed, by construction, the control $u:= \mathcal{M}^T_k \left[ d(\psi_0, \psi_f) \right]$ satisfies the moment problem \eqref{pb_mom_v} after a shift in the indexes,
entailing that the function defined by \eqref{def_right_inv} is a right inverse of $d\Theta_T(\varphi_1,0)$. 
Finally, \cref{thm_mom_weak_estim} also gives that the control satisfies the boundary conditions \eqref{bc} and gives the existence of a constant $C_T>0$ (not depending on $\psi_0$ nor on $\psi_f$) such that, 
\begin{equation*}
\| u \|_{H^m_0(0,T)} \ioe C \| d \|_{h^{2m}(\N, \C)}, \quad \forall m=-(k+1), \dots, k. 
\end{equation*}
Yet, as the function $\mu$ satisfies the hypothesis $\eqref{hyp_mu}$, we get for all $m=-(k+1), \ldots, k$, 
\begin{align*}
\| u \|_{H^m_0(0,T)} &\ioe C \left\| \left( \langle \psi_f , \psi_j(T) \rangle - \langle \psi_0, \varphi_j \rangle \right)  \right\|_{h^{2(m+p)+3}(J, \C)} \\
&\ioe C \left( \| \psi_f \|_{H^{2(p+m)+3}_{(0)}(0,1)} + \|\psi_0\|_{H^{2(p+m)+3}_{(0)}(0,1)} \right).
\end{align*}
\end{proof}

\begin{rem}
The boundary conditions \eqref{bc} allow to ease the propagation of estimates \eqref{size_estimate} to the nonlinear dynamics (see in the following \cref{estim_rq_faible} where those boundary conditions are useful to quantify the error between the nonlinear and the linearized dynamics). Notice that one can't add the boundary condition $u_1(T)=0$ to \eqref{bc} as the term $u_1(T)$ drives the behavior of the first component of the linearized system (see \eqref{expr_lin}).  
\end{rem}

\subsection{Error estimates between the nonlinear and linearized dynamics}
\label{linear_approx}
The proof of \cref{thm:contr_lin_proj} relies on the inverse mapping theorem: the control steering the solution of the Schrödinger equation \eqref{Schrodinger} from $\Psi_0$ to $\Psi_f$ is constructed as the fixed point of the map
 \begin{equation*}
\Phi : (\psi_0, u)  \mapsto  (\psi_0,u) - d \Theta_T(\varphi_1,0)^{-1} . \left[ \Theta_T(\psi_0,u)- (\Psi_0, \Psi_f)\right].  
\end{equation*}
Notice that such function can be rewritten as, 
\begin{equation*}
\Phi \left( \psi_0, u \right) - (\varphi_1,0)
 = -  
 d \Theta_T(\varphi_1,0)^{-1} . \left[
 \Theta_T(\psi_0,u) -  d \Theta_T(\varphi_1,0).(\psi_0- \varphi_1, u) - (\Psi_0, \Psi_f)
 \right].
\end{equation*}
Therefore, estimates on the nonlinear control map $(\Psi_0, \Psi_f) \mapsto u$ are closely linked to estimates on the linear approximation of the end-point map $\Theta_T$, given in the following proposition. 
\begin{prop}
\label{estim_rq_faible}
Let $T>0$, $(p,k) \in \N^2$ with $p \soe k$, $\mu \in H^{2(p+k)+3}(0,1)$ with $\mu^{(2n+1)}(0)=\mu^{(2n+1)}(1)=0$ for all $n=0, \ldots, p-1$, $u \in H^k_0(0,T)$ and $\psi_0 \in H^{2(p+k)+3}_{(0)}(0,1)$. For every $R>0$, there exists a constant $C=C(T, \mu, R)>0$ such that if $\| u \|_{H^k_0(0,T)} < R$ and $u_2(T)=\ldots=u_{k+1}(T)=0$, then the following estimates hold
\begin{multline}
\label{estim_rest_quad}
\left\|
\Theta_T(\psi_0,u) - \Theta_T( \varphi_1, 0)- d \Theta_T(\varphi_1,0).(\psi_0- \varphi_1, u)
\right\|
_{H^{2(p+m)+3}_{(0)} \times H^{2(p+m)+3}_{(0)}}
\\ \ioe 
C
\|u \|_{H^m}
\left(
\| \psi_0 - \varphi_1 \|_{H^{2(p+k)+3}_{(0)}} 
+
\| u \|_{H^k_0}
\right)
,
\quad 
\forall m=-(k+1), \ldots, k.
\end{multline}
\end{prop}
\begin{rem}
The assumption $p \soe k$ appears here to make sure that all the spaces $H^s_{(0)}(0,1)$ involved are positive Sobolev spaces. 
\end{rem}
To prove \cref{estim_rq_faible}, we first show estimates on the solution of the Schrödinger equation in regular spaces. 
\begin{prop}
\label{estim_rest_lin}
Let $T>0$, $(p,k) \in \N^2$, $\mu \in H^{2(p+k)+3}( (0,1), \R)$ with $\mu^{(2n+1)}(0)=\mu^{(2n+1)}(1)=0$ for all $n=0, \ldots, p-1$, $u \in H^{k}_0( (0,T), \R)$, $\psi_0 \in H^{2(p+k)+3}_{(0)}(0,1)$. Then, if $\psi:=\psi( \cdot; \ u, \psi_0)$, 
\begin{equation*}
\psi - \psi_1 \in C^k( [0,T], H^{2p+3}_{(0)}(0,1)) \cap H^{k+1}( (0,T), H^{2p+1}_{(0)}(0,1)). 
\end{equation*}
Moreover, for all $R>0$, there exists a constant $C=C(T, \mu, R)>0$ such that if $\| u \|_{H^k_0(0,T)} < R$ then the following estimates hold
\begin{align}
\label{estim_der_n}
\left\| 
\partial_t^{n} \left( \psi-\psi_1 \right) 
\right\|_{C^0( [0,T], H^{2p+3}_{(0)})} 
&\ioe 
C 
\left(
\| \psi_0 - \varphi_1 \|_{H^{2(p+n)+3}_{(0)}} +
\| u \|_{H^{n}}
\right),
\quad \forall n=0, \ldots, k,
\\
\label{estim_der_k+1}
\left\| 
\partial_t^{k+1} \left( \psi-\psi_1 \right)
\right\|_{L^2( (0,T), H^{2p+1}_{(0)})} 
&\ioe 
C 
\left(
\| \psi_0 - \varphi_1 \|_{H^{2(p+k)+3}_{(0)}} +
\| u \|_{H^{k}}
\right).
\end{align}
\end{prop}

\begin{proof}
Let $\mu$ in $H^{2(p+k)+3}( (0,1), \R)$ with $\mu^{(2n+1)}(0)=\mu^{(2n+1)}(1)=0$ for all $n=0, \ldots, p-1$ and $\psi_0$ in $H^{2(p+k)+3}_{(0)}(0,1)$. Let $R>0$ and $u \in H^k_0(0,T)$ such that $\| u \|_{H^k_0(0,T)}<R$. First, as $\psi$ is the solution of \eqref{Schrodinger}, $\psi-\psi_1$ is the solution of the following Cauchy problem
\begin{equation}  
\label{Schro_rest_lin} 
\left\{
    \begin{array}{ll}
        i \partial_t \left( \psi-\psi_1 \right) = - \partial^2_x \left( \psi-\psi_1 \right) -u(t)\mu(x) \psi , \\
        \left( \psi-\psi_1 \right)(t,0) = \left( \psi-\psi_1 \right)(t,1)=0, \\
        \left( \psi-\psi_1 \right)(0,\cdot) = \psi_0-\varphi_1.
    \end{array}
\right.  
\end{equation}
By \cref{thm:well-posedness}, $\psi-\psi_1$, is in $C^k( [0,T], H^{2p+3}_{(0)}(0,1))$ and there exists $C=C(T, \mu, R)>0$ such that
\begin{equation*}
\| \psi- \psi_1 \|_{C^n( [0,T], H^{2p+3}_{(0)}(0,1))} 
\leq 
C 
\left(
\| \psi_0 - \varphi_1 \|_{H^{2(p+n)+3}_{(0)}(0,1)}
+
\| u  \|_{H^n(0,T)}
\right), \quad \forall n=0, \ldots, k,
\end{equation*}
using that $\psi$ is bounded in $C^k( [0,T], H^{2p+3} \cap H^{2p+1}_{(0)})$. 
To get \eqref{estim_der_k+1}, we differentiate \eqref{Schro_rest_lin} in a distribution sense, using Leibniz formula,
\begin{equation}
\label{eq_diff}
i \partial_t^{k+1} \left(\psi-\psi_1\right) = A \partial_t^{k} \left(\psi-\psi_1\right) - \mu \sum\limits_{n=0}^{k} u^{(n)} \partial_t^{k-n} \psi.  
\end{equation}
Yet, as $\psi-\psi_1$ and $\psi$ belong to $C^k( [0,T], H^{2p+3}_{(0)}(0,1))$, $u$ is in $H^k_0(0,T)$ and the space $H^{2p+1}_{(0)}$ is stable by multiplication by $\mu$, the right-hand side of \eqref{eq_diff} belongs to $L^2( (0,T), H^{2p+1}_{(0)}(0,1))$, giving that 
%
\begin{equation*}
\partial_t^{k+1} \left(\psi-\psi_1\right) \in L^2( (0,T), H^{2p+1}_{(0)}(0,1)),
\end{equation*}
with the following estimate, 
\begin{align*}
\left\| 
\partial_t^{k+1} \left(\psi-\psi_1\right)
\right\|_{L^2( (0,T), H^{2p+1}_{(0)})} 
&\ioe
\| \partial_t^{k} \left(\psi-\psi_1\right)\|_{L^2( (0,T), H^{2p+3}_{(0)})} 
+
\| \mu \|_{H^{2p+1}} \| u \|_{H^{k}} \| \psi \|_{C^{k}( [0,T], H^{2p+1}_{(0)})}
\\
&\leq 
C 
\left(
\| \psi_0 - \varphi_1 \|_{H^{2(p+k)+3}_{(0)}}
+
\| u  \|_{H^k}
\right),
\end{align*}
as $\psi$ is bounded in $C^k( [0,T], H^{2p+3}_{(0)}(0,1))$ and using estimate \eqref{estim_der_n} for $n=k$. 
\end{proof}
Now, we can prove \cref{estim_rq_faible}. 
\begin{proof}[Proof of \cref{estim_rq_faible}.]
By \cref{diff_eq_proj}, 
\begin{equation*}
\Theta_T(\psi_0,u) - \Theta_T( \varphi_1, 0)- d \Theta_T(\varphi_1,0).(\psi_0- \varphi_1, u)
=\left( 0, (\psi- \psi_1-\Psi)(T) \right),
\end{equation*}
where $\psi:= \psi( \cdot; \ u, \psi_0)$ and $\Psi$ is the solution of the linearized system \eqref{diff_u_psi0_T0} around $(\varphi_1, \psi_1)$. Then, $\psi-\psi_1-\Psi$ is the solution of the following Cauchy problem
\begin{equation}  
\label{Schro_rq}
\left\{
    \begin{array}{ll}
        i \partial_t (\psi-\psi_1-\Psi) = - \partial^2_x (\psi-\psi_1-\Psi) -u(t)\mu(x) \left( \psi -\psi_1 \right),\\
        (\psi-\psi_1-\Psi)(t,0)=(\psi-\psi_1-\Psi)(t,1)=0, \\
        (\psi-\psi_1-\Psi)(0, \cdot) = 0.
    \end{array}
\right.  
\end{equation}
Let $R>0$ and a control in $u \in H^k_0(0,T)$ such that $\| u \|_{H^k_0(0,T)} < R$ and $u_2(T)=\ldots=u_{k+1}(T)=0$. 

\medskip \noindent \textit{Step 1: $m \in \{0, \ldots, k\}$.}
Estimate \eqref{estim_sol_fin} on $\psi-\psi_1-\Psi$ and  \eqref{estim_der_n} on $\psi-\psi_1$ give $C>0$ such that 
\begin{align*}
\| (\psi-\psi_1-\Psi)(T) \|_{H^{2(p+m)+3}_{(0)}(0,1)}
\ioe 
C
\left(
\| \psi_0- \varphi_1 \|_{H^{2(p+m)+3}_{(0)}(0,1)} 
+
\| u \|_{H^m(0,T)} 
\right). 
\end{align*}
Therefore, one gets \eqref{estim_rest_quad} for all $m \in \{0, \ldots, k \}$ by inclusion of spaces.

\medskip \noindent \textit{Step 2: $m \in \{-(k+1), \ldots, -1\}$.} For the sake of simplicity, we will only treat the worst case $m=-(k+1)$. The general case for any $m \in \{-(k+1), \ldots, -1\}$ can be proved exactly the same. First, notice that, solving explicitly \eqref{Schro_rq}, the quadratic remainder is given by 
\begin{equation*}
(\psi-\psi_1 -\Psi)(T) 
=i
\int_0^T u(t) e^{-iA(T-t)} \mu \left( \psi-\psi_1 \right)(t) dt.
\end{equation*}
To estimate  $(\psi-\psi_1 -\Psi)(T)$  with respect to $u_{k+1}$, one can compute $k+1$ integrations by parts in time to get, using Leibniz formula,   
\begin{multline*}
(\psi-\psi_1 -\Psi)(T) 
=
iu_1(T) \mu \left( \psi-\psi_1 \right)(T) 
\\+ 
 (-1)^{k+1}i\int_0^T
u_{k+1}(t) 
e^{-iA(T-t)} 
\sum \limits_{n=0}^{k+1} 
(iA)^n 
\left( 
\mu 
\partial_t^{k+1-n} \left( \psi-\psi_1 \right)(t) 
\right) 
dt.
\end{multline*}
First, by definition \eqref{def_norm_faible} of the $H^{-(k+1)}(0,T)$-norm and using \eqref{estim_der_n} as $H^{2p+3}_{(0)}$ is included in $H^{2(p-[k+1])+3}_{(0)}$,
\begin{equation*}
\| 
u_1(T) \mu \left( \psi-\psi_1 \right)(T) 
\|_{H^{2(p-[k+1])+3}_{(0)}}
\ioe 
\| u \|_{H^{-(k+1)}(0,T)} 
\left(
\| \psi_0 - \varphi_1 \|_{H^{2p+3}_{(0)}}
+
\| u \|_{L^2}
\right).
\end{equation*}
Besides, by \cref{estim_rest_lin} and as the space $H^{2p+1}_{(0)}$ is invariant by multiplication by $\mu$, for a.e.\ $t \in (0,T)$, $\mu \left( \psi- \psi_1\right)(t)$ is in $H^{2p+1}_{(0)}$ and so 
$(iA)^n 
\left( 
\mu 
\partial_t^{k+1-n} \left( \psi-\psi_1 \right)(t) 
\right) 
$
is in $H^{2(p-n)+1}_{(0)}$. When $n=0, \ldots, k$, $e^{iAs}$ is an isometry from $H^{2(p-n)+1}_{(0)}$ to $H^{2(p-n)+1}_{(0)}$ (as $2(p-n)+1>0$ because $p\geq k$) and thus the triangular inequality directly gives that 
\begin{align*}
\Big\| 
\int_0^T
u_{k+1}(t) 
e^{-iA(T-t)} 
\sum \limits_{n=0}^{k} 
(iA)^n 
\big( 
&\mu 
\partial_t^{k+1-n} \left( \psi-\psi_1 \right)(t) 
\big) 
dt
\Big\|_{H^{2(p-n)+1}_{(0)}} 
\\
&\ioe
\int_0^T | u_{k+1}(t) |
\sum \limits_{n=0}^{k} 
\| 
\partial_t^{k+1-n} \left( \psi-\psi_1 \right)(t) \|_{H^{2p+1}_{(0)}}dt
\\
&\ioe
\| u_{k+1} \|_{L^2}
\left( 
\| \psi_0 - \varphi_1 \|_{H^{2(p+k)+3}_{(0)}} +
\| u \|_{H^{k}}
\right)
,
\end{align*}
using Cauchy-Schwarz inequality in time and estimates \eqref{estim_der_n}-\eqref{estim_der_k+1} on $\psi-\psi_1$. By inclusion of spaces, such bound still holds when we take the $H^{2(p-[k+1])+3}$-norm of the left hand-side.
For the term $n=0$, for a.e.\ $t \in (0,T)$, $\mu \left( \psi-\psi_1 \right)(t)$ is in $H^{2p+3} \cap H^{2p+1}_{(0)}$, so $A^{k+1} \left(\mu \left( \psi-\psi_1\right) \right)$ is in $H^{2(p-[k+1])+3} \cap H^{2(p-[k+1])+1}_{(0)}$. As $p-[k+1] \geq -1$, \cref{estim_G_C0} gives the existence of a constant $C>0$ such that
\begin{multline*}
\Big\| 
\int_0^T
u_{k+1}(t) 
e^{-iA(T-t)} 
A^{k+1} 
\big( 
\mu 
\left( \psi-\psi_1 \right)(t) 
\big) 
dt
\Big\|_{H^{2(p-[k+1])+3}_{(0)}}
\\
\ioe 
C \| u_{k+1} \|_{L^2} 
\|\mu 
\left( \psi-\psi_1 \right) \|_{C^0( [0,T], H^{2p+3} \cap H^{2p+1}_{(0)})}
\ioe 
C \| u_{k+1} \|_{L^2} 
\left(
\| \psi_0 - \varphi_1 \|_{H^{2p+3}_{(0)}}
+
\| u \|_{L^2}
\right),
\end{multline*}
using the algebra structure of $H^{2p+3}$ and estimate \eqref{estim_der_n} on $\psi-\psi_1$. 
\end{proof}

\subsection{Proof of \cref{thm:contr_lin_proj}: STLC with simultaneous estimates on the control}
\label{main_proof}
Let $p\soe k$ and $\mu$ in $H^{2(p+k)+3}(0,1)$ with $\mu^{(2n+1)}(0)=\mu^{(2n+1)}(1)=0$ for all $n=0, \ldots, p-1$ and satisfying \eqref{hyp_mu}. By continuity with respect to the control (see \cref{thm:well-posedness}), there exists $\widetilde{R}>0$ such that for all $u$ in $B_{\widetilde{R}} \left( H^k_0(0,T) \right)$ and $\psi_0 \in \S \cap H^{2(p+k)+3}_{(0)}(0,1)$
\begin{equation*}
\Re \langle \psi(T; \ u, \psi_0), \psi_1(T) \rangle >0.
\end{equation*}
Let $\eta>0$ so that for all $\psi_f$ in $\S \cap H^{2(p+k)+3}_{(0)}(0,1)$ such that $\| \psi_f -\psi_1(T)\|_{H^{2(p+k)+3}_{(0)}} < \eta$, we have 
\begin{equation*}
\Re \langle \psi_f, \psi_1(T) \rangle >0.
\end{equation*}
Finally, let $R \in (0, \widetilde{R})$.

\medskip \noindent \textit{Step 1: Apply the inverse mapping theorem.}
By \cref{diff_eq_proj} and \ref{inv_droite_con}, the end-point map $\Theta_T$ is $C^1$ on Banach spaces with its differential at $(\varphi_1, 0)$ that admits a continuous right inverse. Therefore, by the inverse mapping theorem (see for example \cite[Subsection 2.3]{BL10} for more details), there exists $\delta \in (0, \eta)$ and a $C^1$-map $\Gamma : \Omega_0 \times \Omega_T \rightarrow H^k_0( (0,T), \R)$ (where $\Omega_0$ and $\Omega_T$ are respectively defined by \eqref{def_Omega_T0} and \eqref{def_Omega_T}) such that for every $(\Psi_0, \Psi_f) \in \Omega_{0} \times \Omega_T$, 
$$
 \P_{\J} \psi(T; \ \Gamma(\Psi_0,\Psi_f), \Psi_0) = \Psi_f. 
$$

\medskip \noindent \textit{Step 2: Gaining the simultaneous estimates and boundary conditions on the nonlinear control.}
To prove that the estimates \eqref{estim_contr_nl}, true at the linear level (see \cref{inv_droite_con}), propagate to the nonlinear system, one must look inside the proof of the inverse mapping theorem.  Let $(\Psi_0, \Psi_f)$ in $\Omega_0 \times \Omega_T$. At Step 1, the inverse mapping theorem gave the existence of $(\psi_0,u) \in \Omega_0 \times B_{R} \left( H^k_0(0,T) \right) $ such that 
\begin{equation*}
\P_{\J} \psi(T; \ u, \psi_0) = \psi_f \quad \text{ and } \quad \Psi_0= \psi_0.
\end{equation*}
This antecedent is constructed as the fixed point of the following application
 \begin{equation*}
 \begin{array}{ccrcl}
\Phi_{ (\Psi_0, \Psi_f)} & : &\Omega_0 \times B_{R} \left( H^k_0(0,T) \right) & \to &\Omega_0 \times B_{R} \left( H^k_0(0,T) \right) \\
 & & (\psi_0, u) & \mapsto & (\psi_0,u) - d \Theta_T(\varphi_1,0)^{-1} . \left[ \Theta_T(\psi_0,u)- (\Psi_0, \Pi_{\psi_1(T)} \Psi_f)\right].  \\
\end{array}
\end{equation*}
and therefore is given by
\begin{equation}
\label{eq:expr_ptfix}
\left( \psi_0, u\right) - (\varphi_1,0)
\\ = -  
 d\Theta_T(\varphi_1,0)^{-1} . \left[
 \Theta_T(\psi_0,u) -  d \Theta_T(\varphi_1,0).(\psi_0- \varphi_1, u) - (\Psi_0, \Pi_{\psi_1(T)} \Psi_f)
 \right].
\end{equation}

\medskip \noindent \textit{Step 2.1: Boundary conditions on the control.}
The linear control map is constructed in \cref{inv_droite_con} so that any linear control defined as $(\psi_0, u):= d \Theta_T(\varphi_1,0)^{-1} ( \psi_0, \psi_f)$ satisfies the boundary conditions \eqref{eq:weak_bc_nl}. Thus from \eqref{eq:expr_ptfix}, one deduces that the nonlinear control also satisfies \eqref{eq:weak_bc_nl}.  
\medskip \noindent \textit{Step 2.2: Simultaneous estimates on the control.}
Equation \eqref{eq:expr_ptfix} together with \cref{inv_droite_con} give the existence of a constant $C=C(T)>0$ such that, for all $m=-(k+1), \ldots, k$, 
\begin{multline*}
\| (\psi_0, u) - (\varphi_1,0) \|_{H^{2(p+m)+3}_{(0)}(0,1) \times H^m_0(0,T)}  
\\ 
\ioe C
\left\|
\Theta_T(\psi_0,u) -  d \Theta_T(\varphi_1,0).(\psi_0- \varphi_1, u) - (\Psi_0, \Pi_{\psi_1(T)} \Psi_f)
\right\|_{H^{2(p+m)+3}_{(0)}(0,1) \times H^{2(p+m)+3}_{(0)}(0,1)}.  
\end{multline*}
Moreover, \cref{estim_rq_faible} gives $C=C(T, \mu, \widetilde{R})>0$ such that, for all $m=-(k+1), \ldots, k$, 
\begin{multline*}
\left\|
\Theta_T(\psi_0,u) - \Theta_T( \varphi_1, 0)- d \Theta_T(\varphi_1,0).(\psi_0- \varphi_1, u)
\right\|
_{H^{2(p+m)+3}_{(0)}(0,1) \times H^{2(p+m)+3}_{(0)}(0,1)}
\\ \ioe 
C
\|u \|_{H^m(0,T)}
\left(
\| \psi_0 - \varphi_1  \|_{H^{2(p+k)+3}_{(0)}(0,1)} 
+
\| u \|_{H^k(0,T)}
\right).
\end{multline*}
Therefore, by the triangular inequality, as $\Theta_T(\varphi_1, 0)=(\varphi_1, \P_{\J} \psi_1(T))$, one gets,
\begin{multline*}
\| (\psi_0, u) - (\varphi_1,0) \|_{H^{2(p+m)+3}_{(0)}(0,1) \times H^m_0(0,T)}
\ioe 
C 
\|u \|_{H^m(0,T)}
\left(
\| \psi_0 - \varphi_1  \|_{H^{2(p+k)+3}_{(0)}(0,1)}  +
\| u \|_{H^k(0,T)}
\right)
\\+
C
\| \Psi_0 - \varphi_1,  \Psi_f - \P_{\J} \psi_1(T)  \|_{H^{2(p+m)+3}_{(0)}(0,1) \times H^{2(p+m)+3}_{(0)}(0,1)}
.
\end{multline*}
Therefore, if $\delta >0$ and $R>0$ are small enough so that, for example, 
\begin{equation*}
C
\left(
\| \psi_0 - \varphi_1 \|_{H^{2(p+k)+3}_{(0)}(0,1)}  +
\| u \|_{H^k(0,T)}
\right) 
\ioe \frac{1}{2}, 
\end{equation*}
one deduces that, as $\Psi_0=\psi_0$ by construction, for all $m=-(k+1), \ldots, k$,  
\begin{equation*}
\|u \|_{H^m(0,T)} 
\ioe 
C
\|  \psi_0 - \varphi_1 ,  \Psi_f - \P_{\J} \psi_1(T)   \|_{H^{2(p+m)+3}_{(0)}(0,1) \times H^{2(p+m)+3}_{(0)}(0,1)},
\quad 
\forall m=-(k+1), \ldots, k.
\end{equation*}
 
\section{Linear control with simultaneous estimates with \cite{EZ10}}
\label{papier_Ervedoza}
The goal of this section is to explain another proof of \cref{inv_droite_con}, relying on the ideas of \cite{EZ10} instead of the moment result given in \cref{thm_mom_weak_estim}. %
\subsection{Building smooth controls for smooth data} 
In \cite{EZ10}, Ervedoza and Zuazua developed a method to construct a control map which preserves the regularity of the data to be controlled for time-reversible linear systems.  However, to use such result in our case, we need to modify slightly the result of \cite{EZ10} as we want to deal with control in projection and not exact controllability. From \cite{EZ10}, we can deduce the following result.
\begin{thm}
\label{thm_EZ10}
Let $X, U$ two Hilbert spaces, $H$ a closed subspace of $X$, $\P$ the orthogonal projection on $H$, $(e^{At})_{t \in \R}$ a strongly continuous group on $X$ with generator $A : D(A) \subset X \rightarrow X$ and $B$ in $\mathcal{L}(U, X_{-1})$ an admissible operator in the sense of \cite[Definition 1.1]{EZ10}. Assume that $\P \circ A= A \circ \P$ and that the system 
\begin{equation}
\label{sys_abstrait}
z'=Az+Bu
\end{equation}
is exactly controllable in projection in some time $T^*$ in $H$: for all $z_0 \in H$, there exists a control $u \in L^2( (0,T^*), U)$ such that the solution of \eqref{sys_abstrait} with the initial condition $z_0$ and control $u$ satisfies 
\begin{equation}
\label{proj}
\P z(T^*)=0.
\end{equation}
Let $T>T^*$, $\delta>0$ such that $T-2\delta \soe T^*$,  $s \in \N_+$ and $ \eta \in C^{s}(\R)$ such that $\eta(t)=0$ if $t \not\in (0,T)$ and $\eta(t)=1$ if $t \in [\delta, T- \delta]$. There exists a constant $C=C(s, T, \eta)>0$ and a linear map $\mathcal{V} : D(A^s) \cap H \rightarrow H^s_0 ( (0,T), U)$ such that for every $z_0 \in D(A^s) \cap H$, the solution of \eqref{sys_abstrait} with control $V:= \mathcal{V}(z_0)$ and initial condition $z_0$ belongs to $C^s([0,T], X)$ and satisfies the requirement \eqref{proj} with
\begin{equation}
\label{estim_abst}
\int_0^T \| V(t)\|^2_U \frac{dt}{\eta(t)} \ioe C \| z_0 \|^2_X
\quad \text{ and } \quad
\forall m \in (0, s), \
\| V \|_{H^{m}_0((0,T), U)} \ioe C \| z_0 \|_{D(A^{m})}.
\end{equation}
\end{thm}
\begin{proof}
This proposition is a consequence of the work \cite{EZ10} noticing that, as $\P \circ A= A \circ \P$, for every $z_0 \in H$ and every control $u \in L^2( (0,T), U)$
\begin{equation}
\label{link_proj}
\P z( \cdot; \ u, z_0) = y( \cdot; \ u, z_0),
\end{equation}
where $z$ is the solution of \eqref{sys_abstrait} and $y$ is the solution of 
\begin{equation}
\label{system_abstrait_proj}
y'= A y + \P \circ B u.  
\end{equation}
Therefore, we apply \cite{EZ10} to the system \eqref{system_abstrait_proj} working on the Hilbert space $H$ endowed with the scalar product of $X$, with the operators defined by $D(\tild{A})= D(A) \cap H$, $\tild{A}=A$ which still generates a strongly continuous group and $\tild{B}= \P \circ B$ still an admissible operator. From \eqref{link_proj}, the exact controllability in projection of \eqref{sys_abstrait} entails the exact controllability of \eqref{system_abstrait_proj}. Therefore, \cite[Proposition 1.3, Theorem 1.4 and Corollary 1.5]{EZ10} give the existence of a linear continuous map $\mathcal{V} : D(\tild{A}^s) \rightarrow H^s_0( (0,T), U)$ with the size estimates \eqref{estim_abst} such that for every $z_0 \in D(\tild{A}^s)=D(A^s) \cap H$, $y( \cdot; \mathcal{V}(z_0), z_0)$ belongs to $C^s( [0,T], H)$ with $y(T; \ \mathcal{V}(z_0), z_0)=0$ and thus $\P z(T; \mathcal{V}(z_0), z_0)=0$ by \eqref{link_proj}.
\end{proof}

\subsection{Proof of \cref{inv_droite_con} with \cite{EZ10}}
The goal of this subsection is to apply \cref{thm_EZ10} to the linearized Schrödinger equation. To that end, we first state that with a change of global phase, one can work with a stationary equilibrium rather than around the ground state. 
%
\begin{rem}
\label{rem:rot}
By \cref{diff_eq_proj}, for all $u \in H_0^k(0,T)$ and $\psi_0 \in H^{2(p+k)+3}_{(0)}(0,1)$, 
 $$d \Theta_T (\varphi_1,0). ( \psi_0, u)= \left(\psi_0,  \P_{\J} \Psi(T) \right),$$
 where $\Psi$ is the solution of the linearized system around the ground state,
 \begin{equation}  
 \label{Schro_lin}
 \left\{
    \begin{array}{ll}
        i \partial_t \Psi(t,x) = - \partial^2_x \Psi(t,x) -u(t) \mu(x) \psi_1(t,x),  \quad (t,x) \in (0,T) \times (0,1),\\
        \Psi(t,0) = \Psi(t,1)=0, \quad t \in (0,T), \\
        \Psi(0,x)=\psi_0. 
    \end{array}
\right.  \end{equation}
To work with a stationary equilibrium, one can perform the change of function $\Psi(t,x)= \tild{\Psi}(t,x) e^{-i \lambda_1 t}$ to work instead with  
 \begin{equation}  
 \label{Schro_lin_rot}
 \left\{
    \begin{array}{ll}
        i \partial_t \tild{\Psi}(t,x) = \left(- \partial^2_x- \lambda_1 \Id \right) \tild{\Psi}(t,x) -u(t) \mu(x) \varphi_1(x),  \quad (t,x) \in (0,T) \times (0,1),\\
        \tild{\Psi}(t,0) = \tild{\Psi}(t,1)=0, \quad t \in (0,T), \\
        \tild{\Psi}(0,x)=\psi_0. 
    \end{array}
\right.  \end{equation}
Such solutions will be denoted $\tild{\Psi}(\cdot; \ u, \psi_0)$. To prove \cref{inv_droite_con}, it is then equivalent to prove that there exists a constant $C_T>0$ such that for all $(\psi_0, \psi_f)$ in $[T_{\S} \varphi_1 \cap H^{2(p+k)+3}_{(0)}] \times [T_{\S} \varphi_1 \cap \H \cap H^{2(p+k)+3}_{(0)}] $, there exists $u \in H_0^k(0,T)$ (constructed as a linear function of $\psi_0$ and $\psi_f$) such that $\P_{\J} \tild{\Psi}(T; \ u, \psi_0)=\psi_f$ with the boundary conditions \eqref{bc} and the estimates \eqref{size_estimate}.
\end{rem}
To apply \cref{thm_EZ10}, we must check that \eqref{Schro_lin_rot} is controllable in projection in an appropriate functional setting. This is done by solving a trigonometric moment problem. 
\begin{prop}
\label{lem:lin_con_L2}
Let $T>0$ and $\mu$ in $H^{2p+3}( (0,1), \R)$ satisfying \eqref{hyp_mu}. The linear equation \eqref{Schro_lin_rot} is exactly controllable in projection in $H^{2p+3}_{(0)}(0,1) \cap T_{\S} \varphi_1$ in time $T$: for all $\psi_0 \in H^{2p+3}_{(0)}(0,1) \cap T_{\S} \varphi_1 $, there exists a control $u \in L^2( (0,T), \R)$ such that 
$$
\P_{\J} \tild{\Psi}(T; \ u, \psi_0)=0. 
$$
\end{prop}

\begin{proof}
Let $\psi_0 \in H^{2p+3}_{(0)}(0,1) \cap T_{\S} \varphi_1$. Solving explicitly \eqref{Schro_lin_rot}, one gets
\begin{equation*}
\tild{\Psi}(T) = 
\sum \limits_{j=1}^{+\infty} \langle \psi_0, \varphi_j \rangle e^{-i (\lambda_j - \lambda_1)T} \varphi_j
+
i \sum \limits_{j=1}^{+\infty}
\left(
\langle \mu \varphi_1, \varphi_j \rangle 
\int_0^T u(t) e^{-i (\lambda_j - \lambda_1)(T-t)} dt 
\right)
\varphi_j.
\end{equation*}
Therefore, the equality $\P_{\J} \tild{\Psi}(T)=0$ is equivalent to the following trigonometric moment problem
\begin{equation*}
\int_0^T u(t) e^{i (\lambda_j- \lambda_1) t} dt = -\frac{\langle \psi_0, \varphi_j \rangle}{i \langle \mu \varphi_1, \varphi_j \rangle}=:d_{j-1}, \quad \forall j \in J.
\end{equation*}
Yet, assumption \eqref{hyp_mu} on $\mu$ together with the fact that $\psi_0$ belongs to $H^{2p+3}_{(0)}(0,1)$ entails that the sequence $(d_j \1_{j \in J})_{j \in \N}$ is in $l^2_r(\N,\C)$ (as $\psi_0$ belongs to $T_{\S} \varphi_1$). Thus, the solvability of a trigonometric moment problem in $L^2(0,T)$ given for example in \cref{thm_mom_L2_poly} concludes the proof. 
\end{proof}
%
%
Then, we can apply \cref{thm_EZ10} to our system.
\begin{lem}
\label{lem_EZ10}
Let $T>0$ and $\mu$ in $H^{2(p+k)+3}( (0,1), \R)$  with $\mu^{(2n+1)}(0)=\mu^{(2n+1)}(1)=0$ for all $n=0, \ldots, p-1$ and satisfying \eqref{hyp_mu}. There exists a linear map $\mathcal{L}: T_{\S} \varphi_1 \cap \H \cap H^{2(p+k)+3}_{(0)}(0,1) \rightarrow H^k_0( (0,T), \R)$ and a constant $C>0$ such that for all $\psi_0$ in $T_{\S} \varphi_1 \cap \H \cap H^{2(p+k)+3}_{(0)}(0,1)$, if $u:=\mathcal{L}(\psi_0)$, 
$$
\P_{\J} \tild{\Psi}(T; \ u, \psi_0 - \langle \psi_0, \varphi_1 \rangle \varphi_1)=0.
$$
Moreover, every control satisfies the boundary conditions 
\begin{equation}
\label{bc_u_inter}
u_1(T)=u_2(T)= \ldots =u_{k+1}(T)=0,
\end{equation}
and the following size estimates
\begin{equation}
\label{size_estimate_u_inter}
\left\| u \right\|_{H^{m}(0,T)} 
\ioe 
C 
\left\|  \psi_0  - \langle \psi_0, \varphi_1 \rangle \varphi_1 \right\|_{H^{2(p+m)+3}_{(0)}(0,1) }, \quad \forall m=-(k+1),\ldots,k.
\end{equation}
\end{lem}

\begin{proof}
First, notice that the equation \eqref{Schro_lin_rot} can be put into the abstract setting 
\begin{equation*}
i \frac{d \tild{\Psi}}{dt}= \widetilde{A} \tild{\Psi} - Bu
\end{equation*}
with 
\begin{align*}
\widetilde{A} \varphi &= \left(A -\lambda_1 \Id \right)\varphi, \quad D(\widetilde{A}):=H^{2p+5}_{(0)}(0,1) \subset X:=H^{2p+3}_{(0)}(0,1) \cap T_{\S} \varphi_1, \\
B u &= u \times \mu \varphi_1, \quad B : \R \rightarrow H^{2p+1}_{(0)}(0,1) 
\end{align*}
The operator $\widetilde{A}$ generates a strongly continuous group on $X$. Let $\psi_0 \in \H \cap H^{2(p+k)+3}_{(0)}(0,1)$ and $\psi_{00}=: \mathcal{C}(\psi_0)$ in $\H \cap H^{2p+4k+5}_{(0)}(0,1)$ the solution of the elliptic equation,
\begin{equation*}
(-i\widetilde{A})^{k+1} \psi_{00} =  \psi_0- \langle \psi_0, \varphi_1 \rangle \varphi_1 
\quad \text{ with } \quad 
\langle \psi_{00}, \varphi_1 \rangle =0. 
\end{equation*}
%
Let $T^* \in (0, T)$, $\delta >0$ such that $T- 2 \delta \geq T^*$ and $\eta: \R \rightarrow [0,1]$ in $C^{2k+1}( \R)$ such that $\eta(t)=0$ if $t \not\in (0,T)$ and $\eta(t)=1$ if $t \in [\delta, T- \delta]$. 
As the system \eqref{Schro_lin_rot} is admissible by \cref{estim_G_C0} and exactly controllable in projection in time $T^*$ by \cref{lem:lin_con_L2}, by \cref{thm_EZ10}, there exists a constant $C>0$ and a control $V=\mathcal{V}(\psi_{00})$ in $H^{2k+1}_0( (0,T), \R) \cap L^2( (0,T), \frac{dt}{\eta})$ (with $\mathcal{V}$ a linear map) such that the solution $\underline{\tild{\Psi}}( \cdot; \ V, \psi_{00})$ of \eqref{Schro_lin_rot} 
belongs to $C^{2k+1}( [0,T], H^{2p+3}_{(0)}(0,1))$ and satisfies $\P_{\J} \underline{\tild{\Psi}}( T; \ V, \psi_{00})=0$ with the following estimates on the control,
\begin{align}
\label{estim_V_1} 
\int_0^T V(t)^2 \frac{dt}{\eta(t)} &\leq C_T \| \psi_{00} \|^2_{H^{2p+3}_{(0)}(0,1)},
\\
\label{estim_V_2}
\| V\|_{H^m_0(0,T)} &\leq C \| \psi_{00} \|_{H^{2(p+m)+3}_{(0)}(0,1)}, \quad \forall \ m=1, \ldots 2k+1. 
\end{align}
Then, if $u:=V^{(k+1)} \in H_0^k(0,T)$, by uniqueness,  $\tild{\Psi}( \cdot; \ u, \psi_0 - \langle \psi_0, \varphi_1 \rangle \varphi_1) = \partial_t^{k+1} \underline{\tild{\Psi}}(\cdot; \ V, \psi_{00})$ and thus satisfies $\P_{\J} \tild{\Psi}(T; \ u, \psi_0 - \langle \psi_0, \varphi_1 \rangle \varphi_1)=(-i \tild{A})^{k+1} \P_{\J} \underline{\tild{\Psi}}(T)=0$. By construction, $u$ satisfies the boundary conditions \eqref{bc_u_inter} and the size estimates \eqref{size_estimate_u_inter} on $u$ are deduced from \eqref{estim_V_1}-\eqref{estim_V_2} as for all $m=0, \ldots, 2k+1$, $\| \psi_{00} \|_{H^{2(p+m)+3}_{(0)}(0,1)}= \| \ \psi_0- \langle \psi_0, \varphi_1 \rangle \varphi_1  \|_{H^{2(p+m-(k+1))+3}_{(0)}(0,1)}$. Therefore, the proof is concluded with $\L := D^{(k+1)} \circ \mathcal{V} \circ \mathcal{C}$ where $D^{(k+1)} : V \mapsto V^{(k+1)}$. 
\end{proof}
As stated in \cref{lem_EZ10}, the change of variables done in \cref{rem:rot} to work with a stationary equilibrium entails that we miss the first coordinate. However, we can correct it in a second time using a moment problem and thus prove \cref{inv_droite_con}. 
\begin{proof}[Proof of \cref{inv_droite_con}.]
We prove \cref{inv_droite_con} using the strategy of \cref{rem:rot}, first for $\psi_f=0$ and then, by the time-reversibility of the system, for any target $\psi_f \in \H$. 

\bigskip
\noindent \emph{Step 1: Proof for $\psi_f=0$.} 
More precisely, we prove that there exists a linear operator $\mathcal{U}$ such that for all $\psi_0$ in $T_{\S} \varphi_1 \cap \H \cap H^{2(p+k)+3}_{(0)}(0,1)$, the control $U:=\mathcal{U}(\psi_0) \in H^k_0(0,T)$ satisfies 
 $
\P_{\J} \tild{\Psi}(T; \ U, \psi_0)=0
$
with the size estimates \eqref{size_estimate} (with $\psi_f=0$) and to prepare Step 2, solving the polynomial moment problem
\begin{equation}
\label{mom_poly_inter}
\int_0^T t^m U(t) dt=0, \quad \forall m=1, \ldots, k.
\end{equation}
By linearity,
\begin{equation*}
\tild{\Psi}(T; \ u+v, \psi_0) =\tild{\Psi}(T; \ u, \langle \psi_0, \varphi_1 \rangle \varphi_1) +\tild{\Psi}(T; \ v, \psi_0-\langle \psi_0, \varphi_1 \rangle \varphi_1).
\end{equation*}
Yet, by \cref{lem_EZ10}, we have $\P_{\J} \tild{\Psi}(T; \ \mathcal{L}(\psi_0), \psi_0-\langle \psi_0, \varphi_1 \rangle \varphi_1)=0$ with every control $\mathcal{L}(\psi_0)$ satisfying the estimates \eqref{size_estimate_u_inter} and the polynomial moment problem \eqref{mom_poly_inter} thanks to the boundary conditions \eqref{bc_u_inter}.
Therefore, to prove Step 1, it is enough to prove the existence of $u \in H^k_0(0,T)$ such that $\P_{\J} \tild{\Psi}(T; \ u, \langle \psi_0, \varphi_1 \rangle \varphi_1)=0$ with the size estimates \eqref{size_estimate} and the polynomial moment problem \eqref{mom_poly_inter}. Besides, solving explicitly \eqref{Schro_lin_rot} with initial condition $\langle \psi_0, \varphi_1 \rangle \varphi_1$, we get
\begin{equation*}
\tild{\Psi}(T; \ u, \langle \psi_0, \varphi_1 \rangle \varphi_1) 
=
\langle \psi_0, \varphi_1 \rangle \varphi_1
+ 
i
\sum \limits_{j=1}^{+\infty} 
\left(
\langle \mu \varphi_1, \varphi_j \rangle 
\int_{0}^T u(t) e^{-i (\lambda_j - \lambda_1)(T-t)} dt 
\right)
\varphi_j. 
\end{equation*}
Thus, the equality $ \P_{\J} \tild{\Psi}(T)=0$ is equivalent to the trigonometric moment problem 
\begin{equation}
\label{pbmom_inter}
\int_0^T u(t) e^{i (\lambda_j - \lambda_1) t} dt = -\frac{ \langle \psi_0, \varphi_1 \rangle \delta_{j,1}}{i \langle \mu \varphi_1, \varphi_j \rangle}, \quad \forall j \in J.  
\end{equation}
By \cref{thm_mom_L2_poly}, if $d(\psi_0):= \left( \left( (-1)^{k+1} k! \frac{\langle \psi_0, \varphi_1\rangle}{i \langle \mu \varphi_1, \varphi_1\rangle} \delta_{q,k} \right)_{q=1, \ldots, 2k}, (0)_{j \in \N^*} \right)$, the control $w:= \L_0^T( d(\psi_0))$ in $L^2(0,T)$ solves the following moment problem,
\begin{align}
\label{moment_0}
&\int_0^T w(t) e^{i (\lambda_j - \lambda_1)t} dt = 0, \quad \forall j \in \N^*, 
\\
\label{moment_der}
&\int_0^T t^q w(t) dt = (-1)^{k+1} k! \frac{\langle \psi_0, \varphi_1\rangle}{i \langle \mu \varphi_1, \varphi_1\rangle} \delta_{q,k}, \quad \forall q=1, \ldots, 2k, 
\end{align}
with the size estimate,
\begin{equation*}
\| w \|_{L^2(0,T)} \ioe C \left| \frac{\langle \psi_0, \varphi_1 \rangle}{i \langle \mu \varphi_1, \varphi_1 \rangle} \right|.
\end{equation*}
Let denote by $u$ the $k$-th primitive of $w$ with vanishing terms at $t=0$,
$$u(t):=\int_0^t \int_0^{t_1} \int_0^{t_2} \ldots \int_0^{t_{k-1}} w(t_k) dt_k dt_{k-1} \ldots dt_1, \quad \forall t \in [0,T],$$
meaning that $u$ solves $u^{(k)}=w$ with $u^{(m)}(0)=0$ for all $m=0, \ldots k-1$.  

\medskip \noindent 
Equation \eqref{moment_0} for $j=1$ and equations \eqref{moment_der} for $q=1, \ldots k-1$ entail that  $u^{(m)}(T)=0$ for all $m=0, \ldots k-1$.  Therefore, $u \in H^k_0(0,T)$ and using Poincaré inequality repeatedly, one gets the existence of $C>0$ such that for all $m=-(k+1), \ldots, k$, 
\begin{equation*}
\| u \|_{H^m_0(0,T)} \ioe C \| u \|_{H^k_0(0,T)}= C \| w \|_{L^2(0,T)} \ioe C | \langle \psi_0, \varphi_1 \rangle |,
\end{equation*}
giving the size estimates \eqref{size_estimate} as we deal with only one moment. Then, performing integrations by parts as $u$ has vanishing boundary terms, one can use \eqref{moment_der} for $q=k$ and \eqref{moment_0} for all $j \soe 2$ to prove that $u$ satisfies the moment problem \eqref{pbmom_inter}. Finally, one can use \eqref{moment_der} for $q=k+1, \ldots, 2k$ to get \eqref{mom_poly_inter}. 

\medskip \noindent 
Therefore, Step 1 is concluded with 
\begin{equation*}
\mathcal{U} := \mathcal{B}_k \circ \L_0^T \circ d + \L,
\end{equation*} 
where $\mathcal{B}_k$ is the operator primitiving $k$ times, given in \eqref{op_primitive}.

\bigskip
\noindent \emph{Step 2: Proof for any target.}
Let $(\psi_0, \psi_f)$ in $[T_{\S} \varphi_1 \cap H^{2(p+k)+3}_{(0)}] \times [T_{\S} \varphi_1 \cap \H \cap H^{2(p+k)+3}_{(0)}]$. By Step 1, there exists $U:=\mathcal{U} \left( \overline{\psi_f}-\overline{ \P_{\J} e^{-i\tild{A}T}\psi_0}\right) \in H^k_0(0,T)$ (where $\mathcal{U}$ is constructed at Step 1) such that
\begin{equation*}
\P_{\J} \tild{\Psi}(T; \ U, \overline{\psi_f}-\overline{ \P_{\J} e^{-i\tild{A}T}\psi_0})=0, 
\end{equation*}
with the polynomial moment \eqref{mom_poly_inter} and the size estimates \eqref{size_estimate} on $U$. Then, if we denote by $\tild{\psi}_{0}:= \overline{\tild{\Psi}}(T; \ U, \overline{\psi_f}-\overline{ \P_{\J} e^{-i\tild{A}T}\psi_0})$ and $V:=U(T- \cdot)$, by uniqueness, 
\begin{equation*}
\tild{\Psi}(t; \ V, \tild{\psi}_{0}) = \overline{\tild{\Psi}}(T-t; \ U, \overline{\psi_f}-\overline{ \P_{\J} e^{-i\tild{A}T}\psi_0}), 
\end{equation*}
and so, 
\begin{equation*}
\P_{\J} \tild{\Psi}(T; \ V, \tild{\psi}_{0}) = \psi_f- \P_{\J} e^{-i\tild{A}T}\psi_0.
\end{equation*}
And, finally, the solution associated with initial condition $\psi_0$ and control $V$ is given by 
\begin{equation*}
\tild{\Psi}(t; \ V, \psi_0) = e^{-i\tild{A}t} \left( \psi_0- \tild{\psi}_{0} \right) + \tild{\Psi}(t; \ V, \tild{\psi}_{0}),
\end{equation*}
and thus, as $\tild{\psi_{0}}$ is in $\H^{\perp}$ by construction, 
\begin{equation*}
\P_{\J} \tild{\Psi}(T; \ v, \psi_0) = \psi_f.
\end{equation*}
Besides, estimates \eqref{size_estimate} hold for $U$ and so for $V$ by translation. Moreover, the polynomial moments of $U$ given in \eqref{mom_poly_inter} entails the boundary conditions \eqref{bc} on $V$ because, for all $m=1, \ldots, k$,
\begin{equation*}
V_{m+1}(T) = \int_0^T \frac{ (T-t)^m}{m!} V(t) dt = \int_0^T \frac{ t^m}{m!} U(t) dt=0. 
\end{equation*}
Therefore, the proof of \cref{inv_droite_con} holds with 
\begin{equation*}
d \Theta_T (\varphi_1, 0) ^{-1}( \psi_0, \psi_f):= \tau_T \circ \mathcal{U}  \left( \overline{\psi_fe^{i \lambda_1 T}}-\overline{ \P_{\J} e^{-i\tild{A}T}\psi_0}\right),
\end{equation*}
where $\tau_T: U \mapsto U(T- \cdot)$ is the translation operator. 
\end{proof}
%
%
%
\begin{rem}
The key point to prove \cref{inv_droite_con} from \cite{EZ10} is \cref{rem:rot} as the work \cite{EZ10} asks to work with a stationary equilibrium. Therefore, it seems like such strategy would not hold when linearizing around a more complicated trajectory than $(\psi_1, u =0)$. For example, when linearizing around a linear combination of trajectories $\psi_j$, for $j \in \N^*$, it would not be straightforward anymore to find a good change of variables allowing us to work equivalently with a stationary equilibrium. That is why in Subsection \ref{C1-regularity}, we gave another proof of \cref{inv_droite_con}, relying on the solvability of a moment problem with simultaneous estimates, giving a strategy that could maybe work when linearizing around other trajectories, if needed. Notice that both strategy rely on the use of a weight function. 
\end{rem}


\bibliographystyle{plain}
\bibliography{biblio_schro_lin}

\end{document}